\documentclass[10pt,a4paper]{article}
\usepackage{amsmath,bm,amssymb,xcolor,amsthm,mathrsfs,cite,hyperref,float}
\usepackage[mathscr]{euscript}
\usepackage{accents}
\usepackage{amsfonts}
\usepackage{booktabs}
\usepackage{multirow}
\usepackage{siunitx}
\usepackage{multirow}
\usepackage{graphicx}
\usepackage[a4paper, bottom=0.6in, top=0.55in, includehead, includefoot, left=1in, right=1in]{geometry}
\usepackage{ upgreek }
\DeclareMathOperator*{\esssup}{ess\,sup}
\theoremstyle{definition}

\newtheorem{theorem}{Theorem}[section]
\newtheorem{corollary}{Corollary}[section]
\newtheorem{lemma}{Lemma}[section]
\newcommand{\bq}{{\bm q}}
\newcommand{\bE}{{\bm e}}
\newcommand{\bz}{{\bm z}}

\newcommand{\bv}{\bm{v}}
\newcommand{\bn}{{\bm n}}
\newcommand{\bQ}{{\bm Q}}
\newcommand{\bV}{{\bm V}}
\newcommand{\bL}{{\bm L}}
\newcommand{\bi}{{\bm \varphi}}
\newcommand{\pV}{{\bm{\Pi}}_{\bf V}}
\newcommand{\bs}{{\theta}}
\newcommand{\br}{{\vartheta}}
\newcommand{\pw}{{ \Pi}_{W}}

\newcommand{\bW}{{\bm W_{22}}}
\newcommand{\bH}{{\bf {H}}}

\numberwithin{equation}{section}
\author{Shipra Gupta 
\thanks{ Department of Mathematics, Birla Institute of Technology and Science, Pilani, Pilani Campus, Vidya Vihar, Pilani Rajasthan 333031 India (p20190427@pilani.bits-pilani.ac.in)}
and Amiya Kumar Pani {\footnote{Corresponding author}}
\thanks{Department of Mathematics, Birla Institute of Technology and Science, Pilani, K.K. Birla Goa Campus, Zuarinagar, Goa-403726, India.
(amiyap@goa.bits-pilani.ac.in)}
and  Sangita Yadav 
\thanks{Department of Mathematics, Birla Institute of Technology and Science, Pilani, Pilani Campus, Vidya Vihar, Pilani Rajasthan 333031 India  (sangita.yadav@pilani.bits-pilani.ac.in) }
}
\title{On Two Conservative HDG Schemes  for Nonlinear Klein-Gordon Equation}
\date{}
\makeatletter
\newcommand{\opnorm}{\@ifstar\@opnorms\@opnorm}
\newcommand{\@opnorms}[1]{%
  \left|\mkern-1.5mu\left|\mkern-1.5mu\left|
   #1
  \right|\mkern-1.5mu\right|\mkern-1.5mu\right|
}
\newcommand{\@opnorm}[2][]{%
  \mathopen{#1|\mkern-1.5mu#1|\mkern-1.5mu#1|}
  #2
  \mathclose{#1|\mkern-1.5mu#1|\mkern-1.5mu#1|}
}
\makeatother
\begin{document}
\maketitle
\begin{abstract}
In this article, a hybridizable discontinuous Galerkin (HDG) method is proposed and analyzed for the Klein-Gordon equation with local Lipschitz-type non-linearity. {\it A priori} error estimates are derived, and it is proved that approximations of the flux and the displacement converge with order $O(h^{k+1}),$ where $h$ is the discretizing parameter and $k$ is the degree of the piecewise polynomials to approximate both flux and displacement variables. After post-processing of the semi-discrete solution, it is shown that the post-processed solution converges with order $O(h^{k+2})$ for $k \geq 1.$ Moreover, a second-order conservative finite difference scheme is applied to discretize in time 
and it is proved that the discrete energy is conserved with optimal error estimates for the completely discrete method. 
To avoid solving a nonlinear system of algebraic equations at each time step, a non-conservative scheme is proposed, and its error analysis is also briefly established. Moreover, another variant of the HDG scheme is analyzed, and error estimates are established. Finally, some numerical experiments are conducted to confirm our theoretical findings.
\end{abstract}

\medskip

\section{Introduction}

This paper focuses on two  HDG methods for the following non-linear Klein-Gordon equation: 
\begin{equation}\left.
\begin{array}{rl}
u_{tt}-\Delta u+f(u)=0&\text{in }\Omega\times(0,T],\\
u=0&\text{on }\partial\Omega\times(0,T],\\
u|_{t=0}=u_0,&u_t|_{t=0}=u_1\text{ in }\Omega,
\end{array}\right\}\label{main}
\end{equation}
where $\Omega$ is a convex polygonal domain in $\mathbb{R}^2$ with boundary $\partial\Omega$, $f(u)$ is a given nonlinear function, and $u_0,~ u_1$ are also given functions. 
When $f(s) = s^3-s$, the problem \eqref{main}, popularly known as the Klein-Gordon equation, occurs as a mathematical model in relativistic quantum mechanics \cite{kragh1984equation} and also in order-disorder transition in solid mechanics \cite{currie1980statistical}. Moreover, when $f(s)= \sin s$, \eqref{main} is called sine-Gordon equation and
 has a wide range of applications in the propagation of fluxons in Josephson junctions between two superconductors, the motion of a rigid pendulum coupled to a stretched wire, and fluid motion stability. For the existence of a unique weak solution of problem \eqref{main}, we refer to \cite{tourigny1990product,kapitanski1994global}. 
 
Earlier literature on numerical methods on finite difference methods for the nonlinear Klein-Gordon equation can be found in \cite{jimenez1990analysis,duncan1997sympletic,wong1997initial,achouri2019finite} and references therein. Tourigny \cite{tourigny1990product} has applied a conforming finite element method in spatial direction combined with product approximation for the nonlinear term and has discussed the optimal rate of convergence for the semi-discrete case. 
In the context of discontinuous Galerkin (LDG) schemes applied to  the two-dimensional sine-Gordon equation on Cartesian grids,
Baccouch  \cite{baccouch2019optimal} 
has derived error estimates and superconvergence results. 

Now, we briefly discuss the HDG methods, which is the theme of this paper. The HDG methods have been introduced in 
\cite{cockburn2009unified,cockburn2009derivation}  in the context of steady-state diffusion  problem, their  convergence and
super-convergence properties in \cite{cockburn2012conditions} as well as their efficiently
implemented in \cite{kirby2012}. These  schemes provide approximations that are more accurate than the
ones given by any other DG methods for second-order elliptic problems \cite{cockburn2010hybridizable}.
The researchers in \cite{cockburn2019interpolatory,chen2019superconvergent,sanchez2021priori,sanchez2022error} have addressed the nonlinear source term using the HDG method for semilinear elliptic and parabolic problems, essentially with globally Lipschitz non-linearity. With the help of the interpolatory HDG approach, the nonlinear component has been dealt with efficiently, and the optimal convergence rate has been proved. For a class of DG methods and with the aid of appropriate post-processing of the semi-discrete potential, Pani and Yadav \cite{yadav2019superconvergent} have derived superconvergence for nonlinear parabolic problems. Concerning a second-order linear wave equation, Cockburn {\it et al.}  \cite{cockburn2018stormer} have devised the energy-conservative HDG methods and have derived uniform in-time estimates along with the superconvergence result of the post-processed solution. For a recent survey on HDG methods, see \cite{cockburn2023}.
 
In most of the papers discussed earlier on numerical approximations to the Klein-Gordon equation, except \cite{tourigny1990product}, the authors have derived their results for a global Lipschitz non-linearity. Since the non-linearity present in the problem \eqref{main} is of polynomial type, and in fact locally  Lipschitz type, as a first step, we derive  {\it a priori}  bounds, which helps to deal with the nonlinearity in the discrete problem. This study builds on two conservative HDG techniques for the nonlinear Klein-Gordon equation with odd-degree non-linearity and achieves discrete energy conservation. The major contributions of this paper are summarized below. 
  \begin{itemize}
      \item {\it A priori} bound in  $L^p$ norm for the discrete displacement variable in the semidiscrete HDG method is derived, and it is exploited to deal with the local Lipschitz non-linearity in the discrete system.
      \item  Convergence of order $O(h^{k+1}),$ is obtained when piece-wise polynomials of degree $k$ approximate both displacement and the flux variables. 
Subsequently, using a variant of Baker's technique, the optimal rate of convergence for the displacement is established under a less smooth regularity assumption on the exact solution.
      \item Low-cost local postprocessing for the displacement has helped us to improve $L^{\infty}(L^2)$ norm convergence, which is of order $k + 2$.
      \item Fully discrete conservative and non-conservative schemes are proposed using the discrete-time Galerkin method, and corresponding error estimates are established.
      \item A variant of the HDG method based on \cite{qiu2018hdg} is analyzed, and error estimates are derived.
      \item The HDG approach has been numerically tested with various degrees of polynomials. Numerical results show that the displacement and flux converge with optimal order, while the post-processed solution achieves superconvergence. These results confirm our theoretical findings. 
  \end{itemize}
  
  This paper is structured as follows. Section $2$ introduces the semi-discrete HDG approach, provides some preliminary results and assumptions, and shows that the method is conservative. Section $3$ focuses on a priori error analysis for approximating the HDG solution. The post-processing is covered in  Section $4$. Section $5$ deals with fully discrete conservative and non-conservative schemes and their error analyses. Section $6$ demonstrates the numerical results and concluding remarks. 
  
  Throughout this paper, whenever there is no confusion, the standard Sobolev spaces $H^m(\Omega)$ will be denoted by $H^m$
  with norm $\|\cdot\|_{H^m}.$ Moreover, $C$ denotes a generic positive constant, which may vary from step to step.

\section{HDG Formulation}
For HDG formulation, the original problem \eqref{main} is first rewritten in mixed form as a system of first-order differential equations by adding an auxiliary flux variable $\bq$. This results in:

\begin{equation*}
\bq=\nabla u \text{ in }\Omega\times(0,T].
\end{equation*}
Then, \eqref{main} becomes
\begin{equation*}
u_{tt}-\nabla\cdot \bq+f(u)=0, \text{ in }\Omega\times(0,T].
\end{equation*}
\subsection{Discretization}
Let $\mathscr{T}_h$ be a shape regular triangulation of $\bar{\Omega}$ with $h_K=\text{diam}(K)~\forall~ K\in\mathscr{T}_h$ and $h=\max\limits_{K\in\mathscr{T}_h}h_K$. Let $\mathcal{E}_h$ be the set of the faces $E$ of elements $K$ of the triangulation $\mathscr{T}_h$, with $\mathcal{P}_k(K)$  as the space of polynomials of degree $\leqslant k$ on $K$. Set 
 
\begin{align*}
W_h&=\{w\in L^2(\Omega):w|_K\in \mathcal{P}_k(K)~\forall ~K\in \mathscr{T}_h\},\\
\bV_h&=\{\bv\in[L^2(\Omega)]^2:\bv|_K\in [\mathcal{P}_k(K)]^2~\forall ~K\in \mathscr{T}_h\},\\
M_h&=\{\mu\in L^2(\mathcal{E}_h):\mu|_{E}\in \mathcal{P}_k({E})~\forall ~{E}\in \mathcal{E}_h\}.
\end{align*}
Define
\[ L^2(\mathscr{T}_h) = \prod_{K \in \mathscr{T}_h} L^2(K),~~ L^2(\mathcal{E}_h)=\prod_{E \in \mathcal{E}_h}L^2({E}),\] 
with inner-product and norm, respectively, as
\begin{align*}
    (\textbf{f},\textbf{g})&=\sum_{K\in\mathscr{T}_h}(\textbf{f},\textbf{g})_K\; ,\quad (\phi,\psi)=\sum_{K\in\mathscr{T}_h}(\phi,\psi)_K\;,\quad\langle\phi,\psi\rangle=\sum_{K\in\mathscr{T}_h}\langle\phi,\psi\rangle_{\partial K}\;,\\
    \langle \phi,\psi\rangle_{\mathcal{E}_h/\Gamma_\partial}&=\sum_{E\in \mathcal{E}_h/\Gamma_\partial}\int_{E}\phi\psi \;\quad\text{ and } \quad\langle \phi,\psi\rangle_{\Gamma_\partial}=\sum_{E\in\Gamma_\partial}\int_{E}\phi\psi\;,
\end{align*}
where, \[(\textbf{f},\textbf{g})_K=\int_K\textbf{f}\cdot\textbf{g}\; ,\quad (\phi,\psi)_K=\int_K\phi\psi\;,\quad \langle\phi,\psi\rangle_{\partial K}=\int_{\partial K}\phi\psi \;.\]

 For vector valued  square integrable space $\bL^2(\mathscr{T}_h) = (L^2(\mathscr{T}_h))^d$, the inner-product and norm are similarly defined.
 
Now, the HDG semi-discrete formulation  is for $t\in(0,T]$ to seek $(u_h(t),\bq_h(t),\widehat{u}_h(t))\in W_h\times \bV_h\times M_h$ such that
\begin{subequations}
\begin{eqnarray}
(\bq_h,\bv_h)+(u_h,\nabla\cdot\bv_h)-\langle\widehat{u}_h,\bv_h\cdot\bn\rangle&=&0\quad\forall \bv_h\in\bV_h,\label{hdg1}\\
(u_{htt},w_h)+(\bq_h,\nabla w_h)-\langle\widehat{\bq}_h\cdot\bn,w_h\rangle+(f(u_h),w_h)&=&0\quad\forall w_h\in W_h,\label{hdg2}\\
\langle \widehat{\bq}_h\cdot\textbf{n},\mu_1\rangle_{\mathcal{E}_h\setminus\Gamma_\partial}&=&0\quad\forall\mu_1\in M_h,\label{hdg3}\\
\langle \widehat{u}_h,\mu_2\rangle_{\Gamma_\partial}&=&0\quad\forall\mu_2\in M_h,\label{hdg4}\\
\widehat{\bq}_h\cdot \bn=\bq_h\cdot\bn-\tau(u_h-\widehat{u}_h) &&\text{ on } \mathcal{E}_h \label{hdg5}
\end{eqnarray}
\end{subequations}
with initial conditions:
\begin{equation}\label{hgd-ICs}
u_h(0)= u_{h0} \in W_h\;\;\;\mbox{ana}\; u_{ht}(0)= u_{h1}\in W_h,
\end{equation}
where $u_{h0}$ and $u_{h1}$ are to be defined later with the following approximation property:
\begin{equation}\label{IC-approx}
\|u_0-u_{h0}\| + \|u_1-u_{h1}\| =O(h^{k+1}).
\end{equation}
Here, $\Gamma_\partial= \mathcal{E}_h\cap\partial\Omega $ is the collection of exterior edges of the domain $\Omega$, $\tau:\mathcal{E}_h\rightarrow\mathbb{R}^+$ is constant on each face of $K\in \mathscr{T}_h$ and $\bn$ is the unit outward normal to element $K$. 
We define the faces norm as \[\|\psi-\widehat\psi\|_+^2= \sum_{E\in\mathscr{E}_h} \|\tau^{1/2}(\psi-\widehat \psi)\|^2_{L^2(E)}\] 
\subsection{HDG Projections}
Now introduce HDG projections $\Pi_W:L^2(K)\rightarrow\mathcal{P}_k(K)$ and ${\boldsymbol \Pi}_{V}:[L^2(K)]^2\rightarrow[\mathcal{P}_k(K)]^2$ such that for any $u\in L^2(K)$ and $\bq\in [L^2(K)]^2$ define $(\Pi_W u,{\boldsymbol \Pi}_{V}\bq)$ as
\begin{subequations}\label{Pj1}
\begin{eqnarray}
(\pV\bq,\bv)_{K}=(\bq,\bv)_{K}&&\text{for all } \bv\in [\mathcal{P}_{k-1}(K)]^2, \label{Pj11}\\
(\pw u,w)_{K}=(u,w)_{K}&&\text{for all } w\in \mathcal{P}_{k-1}(K), \label{Pj12}\\
\langle \pV\bq\cdot\bn-\tau\pw u,\mu\rangle_{E}=\langle \bq\cdot\bn-\tau u,\mu\rangle_{E} && \text{for all }\mu\in\mathcal{P}_{k}(E).\label{Pj13}
\end{eqnarray}
\end{subequations}
This projection $(\pV,\pw)$ as shown in the Appendix $A$ of \cite{cockburn2010projection} is well defined and has reasonable approximation properties, which are stated below in terms of a Lemma~\ref{tt1}.

Based on these projections, we define the quantities 
\begin{align}
    \br_u=\pw u-u,~\br_{\bq}=\pV\bq-\bq,~\widehat{\br}_u= \mathcal{P}u-u,
\end{align} 
where $\mathcal{P}$ is the $L^2$ projection onto $M_h$. 

\begin{lemma}\label{tt1}
Suppose $k \geqslant 0,~\tau|_{\partial K}$ is non-negative and $\tau_{K}^{\max}$:=  $\max\tau|_{\partial K}$ $>$ 0. Then, the system \eqref{Pj1} is uniquely solvable for $(\pV, \pw)$. Furthermore, there is constant $C$ independent of $K$ and $\tau$ such that
\begin{align*}
\lVert\br_{\bq}\rVert_{[L^2(K)]^2} \leqslant Ch_{K}^{l_{q}+1} |\bq|_{{H^{l_{q}+1}}(K)} + Ch_{K}^{l_{u}+1} \tau_{K}^{*}| u |_{{H^{l_{u}+1}}(K)},\\
\lVert\br_u\rVert_{L^2(K)} \leqslant Ch_{K}^{l_{u}+1}| u |_{{H^{l_{u}+1}}(K)} + C\frac{h_{K}^{l_{q}+1}}{\tau_{K}^{\max}} |\nabla . \bq|_{H^{l_{q}}(K)},
\end{align*}
for $l_{u},~l_{q}$ in $[0,k]$, here $\tau_{K}^{*}$ :=  $\max\tau|_{\partial K/ {E}^{*}}$, $ {E}^{*}$ is a face of $K$ at which $\tau|_{\partial K}$ is maximum. Hence, if $\tau_{K}^{*}$
and $\dfrac{1}{\tau_{K}^{\max}}$ are uniformly bounded for all $K \in \mathscr{T}_h$, then the projection converges with order $k+1$ for both variables provided the solution is smooth enough.
\end{lemma}

\subsection{Energy Conservation Property}

This subsection deals with the discrete conservation property.
Set $F(s)=\frac{1}{4}(1-s^2)^2$ and  observe that, $F(s)\geqslant 0$ and $F^{\prime}(s)=f(s)$ for all $s\in\mathbb{R}$. Below, we present the energy conservation property.
\begin{theorem}\label{2.1}
The HDG approximation $(u_h(t),\bq_h(t),\widehat{u}_h(t))\in W_h\times \bV_h\times M_h$ given by \eqref{hdg1}-\eqref{hdg4} satisfies the following discrete conservation  property 
\begin{eqnarray}
\mathcal{E}(t)=\mathcal{E}(0),\ \ \forall \ t\in[0,T]\label{th1-1}
\end{eqnarray}
where
\begin{eqnarray*}
\mathcal{E}(t)&=&\dfrac{1}{2}\big(\|u_{ht}(t)\|^2+\|\bq_h(t)\|^2+\|u_h(t)-\widehat{u}_h(t)\|_+^2+2(F(u_h(t)),1)\big).
\end{eqnarray*}
\end{theorem}
\begin{proof}

Differentiate \eqref{hdg1} with respect to $t$. Then, choose $\bv_h=\bq_h$ in the resulting equation and $w_h=u_{ht}$ in \eqref{hdg2}. On adding these two resulting equations, we arrive at
\begin{equation}
\frac{1}{2}\frac{d}{dt}\left(\|u_{ht}\|^2+\|\bq_h\|^2+2(F(u_h),1)\right)+(u_{ht},\nabla\cdot \bq_h)+(\bq_h,\nabla u_{ht})-\langle\widehat{u}_{ht},\bq_h\cdot\bn\rangle-\langle\widehat{\bq}_h\cdot\bn,u_{ht}\rangle=0.\label{pf1-1}
\end{equation} 
Now, the use of integration by parts shows
\begin{equation}
(u_{ht},\nabla\cdot\bq_h)+(\bq_h,\nabla u_{ht})=\langle u_{ht},\bq_h\cdot\bn\rangle.\label{pf1-2}
\end{equation}
An application of \eqref{hdg3} with \eqref{hdg4} and \eqref{pf1-2} yields
\begin{equation}
(u_{ht},\nabla\cdot \bq_h)+(\bq_h,\nabla u_{ht})-\langle\widehat{u}_{ht},\bq_h\cdot\bn\rangle-\langle\widehat{\bq}_h\cdot\bn,u_{ht}\rangle=\frac{1}{2}\frac{d}{dt}\|u_h-\widehat{u}_h\|_+^2.\label{pf1-3}
\end{equation}
On substitution of \eqref{pf1-3} in \eqref{pf1-1}, we obtain
\[\frac{d}{dt}\mathcal{E}(t)=0.\]
After integrating from $0$ to $t$, we obtain that
$\mathcal{E}(t)=\mathcal{E}(0)\ \ \forall\ t\in(0,T].$
This completes the proof of \eqref{th1-1}. 
\end{proof}
\noindent
The next corollary is a consequence of the discrete energy conservation property. Note that from \eqref{IC-approx}, it follows for a positive constant $C$ that
\begin{equation}\label{IC-stability}
\|u_{h0}\| \leqslant C\;\|u_0\|\;\;\;\mbox{and}\;\;\|u_{h1}\| \leqslant C\;\|u_1\|.
\end{equation}

 \begin{corollary}\label{coro}
There is a positive constant $C$ such that for $t>0$,
\begin{eqnarray*}
\|u_{ht}(t)\|^2+ \|\bq_h(t)\|^2+\|u_h-\widehat{u}_h\|_+^2+\|u_h(t)\|^4_{L^4}\leqslant C.
\end{eqnarray*}
\end{corollary}
\begin{proof}
    From \eqref{th1-1}, 
    it is easy to note  using \eqref{IC-stability}   that
    \begin{align*}
        \|u_{ht}(t)\|^2+ \|\bq_h(t)\|^2+\|u_h-\widehat{u}_h\|_+^2+\|u_h(t)\|^4_{L^4} \leqslant &\mathcal{E}(t)= \mathcal{E}(0)\\ 
    \leqslant &\|u_{ht}(0)\|^2+ \|\bq_h(0)\|^2+\|u_h(0)-\widehat{u}_h(0)\|_+^2\\& +\frac{1}{2} ((1-u^2_h(0))^2,1) \\
    \leqslant& C\;\Big(\|u_1\|^2 + \|\nabla u_0\|^2 + |\Omega|+ \|u_0\|^{4}_{L^4}\Big).
  \end{align*}
This completes the rest of the proof.
\end{proof}
\section{ Semidiscrete Error Estimates}
\subsection{\texorpdfstring{$L^{p}$}{TEXT} Estimate for the Discrete solution }
\begin{lemma}\label{ll2}
  For given $w_h \in W_h$, $\widehat{w}_h \in M_h$ and $\textbf{y}\in (L^2(\Omega))^2$, let $\bz\in (L^2(\Omega))^2$ satisfy
\begin{eqnarray}
(\bz,\bv_h)+(w_h,\nabla\cdot\bv_h)-\langle\widehat{w}_h,\bv_h\cdot\bn\rangle&=&(\textbf{y},\bv_h)\label{p1} 
\end{eqnarray}
for all $\bv_h\in \bV_h$. Then, there exists a positive constant $C$ independent of $h$ such that for $p\in [2,\infty)$
\begin{eqnarray*}
\|w_h\|_{L^{p}} \leqslant C \Big( \|\bz\| + \|w_h-\widehat{w}_h\|_++\|\textbf{y}\| \Big),
\end{eqnarray*}
and for $p=\infty$
\begin{eqnarray*}
\|w_h\|_{L^{\infty}} \leqslant C\; \log \frac{1}{h}\;\Big( \|\bz\| + \|w_h-\widehat{w}_h\|_++\|\textbf{y}\| \Big),
\end{eqnarray*}
\end{lemma}\label{lp}
\begin{proof}
For $g \in L^q(\Omega), ~ \frac{1}{p}+\frac{1}{q} = 1$, let $\psi \in W^{2,q}(\Omega)$ be a solution of the elliptic problem
\begin{eqnarray*}
-\Delta \psi=g \text{ in }\Omega \text{ and } \psi=0 \text{ on }\partial\Omega
\end{eqnarray*}
 satisfying the regularity  $$\|\psi\|_{W^{2,q}}\le C\|g\|_{L^q}.$$
With $\Psi = \nabla \psi$, 
 and using the commutative property of HDG projection (from \cite{cockburn2010projection}),  it follows that 
\begin{eqnarray*}
(w_h,g)& =& -(w_h,\nabla\cdot \Psi) =- (w_h,\nabla \cdot \pV \Psi) +\langle w_h,\tau(\pw \psi-\psi)\rangle\\
        &=& (\bz,\pV\Psi)-\langle \widehat{w}_h,\pV \Psi\cdot\bn\rangle+\langle w_h,\tau(\pw\psi-\psi)\rangle-(\textbf{y},\pV\Psi)\quad{[\text{using }\eqref{p1}]}\\
        &=& (\bz,\pV\Psi)+\langle\tau(w_h-\widehat{w}_h),\pw \psi-\psi\rangle-\langle \widehat{w}_h,\pV\Psi\cdot\bn-\tau(\pw\psi-\psi)\rangle-(\textbf{y},\pV\Psi)\\
        &=& (\bz,\pV \Psi)+\langle\tau(w_h-\widehat{w}_h),\pw \psi-\psi\rangle-\langle \widehat{w}_h, \Psi\cdot\bn\rangle-(\textbf{y},\pV\Psi).\quad {[\text{using }\eqref{Pj13}]}
\end{eqnarray*}
As $\widehat{w}_h$ is single valued, $\Psi \in H(div,\Omega)$ and $\widehat{w}_h=0$ on $\partial\Omega$, $\langle\widehat{w}_h,\Psi\cdot\bn\rangle=0$. Thus, an application of Theorem \ref{tt1} implies
\begin{eqnarray*}
(w_h,g)&=& (\bz,\pV \Psi) +\langle\tau(w_h-\widehat{w}_h),\pw \psi-\psi\rangle-(\textbf{y},\pV\Psi)\\
      &\leqslant& \|\bz\| |\Psi|_1+C\;(h_{k}^{3/2}\tau^{-1/2})\|w_h-\widehat{w}_h\|_+ \|\psi\|_2+\|\textbf{y}\| |\Psi|_1\\
      &\leqslant& C \big(\|\bz\|+\|w_h-\widehat{w}_h\|_++\|\textbf{y}\|\big)\|\psi\|_{2,q}\\
       &\leqslant& C \big(\|\bz\|+\|w_h-\widehat{w}_h\|_++\|\textbf{y}\|\big)\|g\|_{L^{q}}.
\end{eqnarray*}
This implies \[
\|w_h\|_{L^p}=\sup\limits_{\text{\footnotesize$\begin{array}{l}g\in L^{q}\\g\neq0\end{array}$}}\frac{(w_h,g)}{\|g\|_{L^q}}\le \mathrm{C} \big(\|\bz\|+\|w_h-\widehat{w}_h\|_++\|\textbf{y}\|\big).
\]
Now a use of Sobolev inequality when the dimension $d=2$ (see, \cite [page 47]{JT})
$$\|w_h\|_{L^{\infty}} \leqslant C\; \log \frac{1}{h}\; \|w_h\|_{L^p}$$
now completes the rest of the proof.
\end{proof}
\subsection{{\it A priori} Error Estimates}
This subsection deals with {\it a priori} error estimates of the semi-discrete method. 
Decompose, $\bE_u = \bs_u-\br_u,~\bE_\bq=\bs_\bq-\br_{\bq},~\widehat{\bE}_{u}=\widehat{\bs}_{u}-\widehat{\br}_{u}$, where $\bs_u = \pw u-u_h,~\bs_{\bq}=\pV \bq-\bq_h,~\widehat{\bs}_{u}=\mathcal{P}u-\widehat{u}_h$. Since the estimates of $\br_u,~ \br_{\bq},~ \widehat{\br}_{u}$ are known, it is enough to estimate $\bs_u, \bs_{\bq}$ and $\widehat{\bs}_{u}$. Using the HDG projections (\eqref{Pj11}-\eqref{Pj13}) and (\eqref{hdg1}-\eqref{hdg3}), equations in $\bs_u, ~\bs_\bq$ and $\widehat\bs_{u}$ are written  for all $(\bv_h,w_h,\mu_1, \mu_2)\in  \bV_h\times W_h\times M_h\times M_h$ as
\begin{subequations}\label{ereq}
\begin{eqnarray}
(\bs_\bq,\bv_h)+(\bs_u,\nabla\cdot\bv_h)-\langle \widehat{\bs}_u,\bv_h\cdot\bn\rangle&=&(\br_{\bq},\bv_h),\label{ee1}\\
(\bs_{u_{tt}},w_h)+(\bs_\bq,\nabla w_h)-\langle\widehat{\bs}_{\bq}\cdot\bn,w_h\rangle+(f(u)-f(u_h),w_h)
&=&(\br_{u_{tt}},w_h),\label{ee2}\\
\langle \widehat{\bs}_\bq\cdot\bn,\mu_1\rangle_{\mathcal{E}_h\setminus\Gamma_\partial}&=&0,\label{ee4}\\
\langle\widehat{\bs}_u,\mu_2\rangle_{\Gamma_\partial}&=&0,\label{ee5}
\end{eqnarray}
\end{subequations}
where the numerical trace for flux is defined on $\mathcal{E}_h$  as
\begin{align}\label{eflux}
    \widehat{\bs}_{\bq}\cdot\bn=\bs_{\bq}\cdot\bn-\tau(\bs_u-\widehat{\bs}_u).
\end{align}
Following \cite{cockburn2018stormer}, we now  define  initial approximations via the HDG formulation of the associated elliptic equation at $t=0$, i.e., seeking $(u_h(0),\bq_h(0),\widehat{u}_h(0))\in W_h\times \bV_h\times M_h$ as the solution of
\begin{subequations}\label{HDG0}
\begin{eqnarray}
(\bq_h(0),\bv_h)+(u_h(0),\nabla\cdot\bv_h)-\langle\widehat{u}_h(0),\bv_h\cdot\bn\rangle&=&0\quad\forall \bv_h\in\bV_h,\label{hdg10}\\
(\bq_h(0),\nabla w_h)-\langle\widehat{\bq}_h(0)\cdot\bn,w_h\rangle&=&(\Delta u_0, w_h)\quad\forall w_h\in W_h,\label{hdg20}\\
\langle \widehat{\bq}_h(0)\cdot\textbf{n},\mu\rangle_{\mathcal{E}_h\setminus\Gamma_\partial}&=&0\quad\forall\mu\in M_h,\label{hdg30}\\
\langle\widehat{u}_h(0),\mu\rangle_{\Gamma_\partial}&=& 0\quad\forall\mu\in M_h,\label{hdg40}\\
\widehat{\bq}_h(0)\cdot \bn=\bq_h(0)\cdot\bn&-&\tau(u_h(0)-\widehat{u}_h(0)) \text{ on } \;\mathcal{E}_h.\label{hdg50}
\end{eqnarray}
\end{subequations}
Moreover, the initial approximation  $u_{ht}(0)$ of the initial velocity is obtained by differentiating the HDG projection  \eqref{Pj1} concerning time and evaluating at $t=0$ using $(\pV (\nabla u_1), \pw u_1)$ of $(\nabla u_1, u_1).$

For our subsequent analysis, we use the following notation. For a Banach space $X$ with norm $\|\cdot\|_{X}$, we define 
$W^{m,p}(0,T; X),\; 1\leqslant p\leqslant \infty$  as the space of all strongly measurable and Bochner integrable functions $\phi: (0,T)\longrightarrow X$ such that  the corresponding norm  $\|\phi\|_{W^{m,p}(0,T;X)}$ is finite, that is, for $p\in [1,\infty)$  
$$\|\phi\|_{W^{m,p}(0,T;X)}:= \Big(\sum_{j=0}^m \int_{0}^T \|\frac{d^j \phi}{dt^j} (s)\|^p_{X}\;ds\Big)^{1/p}$$
and for $p=\infty$
$$\|\phi\|_{W^{m,\infty}(0,T;X)}:= \max_{0\leqslant j\leqslant m} \esssup_{0<t<T} \|\frac{d^j \phi}{dt^j}(t) \|_{X}.$$

The following Lemma is useful for our subsequent use.
\begin{lemma}\label{IC-approx-1} 
Let  $u_{h0}$ be defined by HDG elliptic problem \eqref{HDG0} and $u_{h1}= \pw u_1.$  Then, there holds
\begin{equation}\label{estimate:IC}
\Big( \|\bs_{\bq}(0)\|^2 + \|(\bs_{u}-\widehat{\bs}_{u})(0)\|_+^2\Big)^{1/2} \leqslant C\;h^{k+1} \;\|u_0\|_{H^{k+2}}.
\end{equation}
Moreover, an additional  elliptic $H^2\cap H^1_0$- regularity  result of the associated adjoint problem in a duality argument  as in \cite[Lemma 4.1]{cockburn2010projection}  shows  
\begin{equation}\label{estimate:IC-0}
\|\bs_{u}(0)\| \leqslant C\;h^{k+2} \;\|u_0\|_{H^{k+2}}.
\end{equation}
\end{lemma}

\begin{proof}
Using HDG projections \eqref{Pj1} at $t=0$ and \eqref{HDG0}, we now arrive at
error equations at $t=0$ as
\begin{subequations}\label{HDG1}
\begin{eqnarray}
(\bs_{\bq}(0),\bv_h)+(\bs_u(0),\nabla\cdot\bv_h)-\langle\widehat{\bs}_u(0),\bv_h\cdot\bn\rangle&=&0\quad\forall \bv_h\in\bV_h,\label{hdg11}\\
(\bs_\bq(0),\nabla w_h)-\langle\widehat{\bs}_\bq(0)\cdot\bn,w_h\rangle&=&0\quad\forall w_h\in W_h,\label{hdg21}\\
\langle \widehat{\bs}_\bq(0)\cdot\textbf{n},\mu\rangle_{\mathcal{E}_h\setminus\Gamma_\partial}&=&0\quad\forall\mu\in M_h,\label{hdg31}\\
\langle\widehat{\bs}_u(0),\mu\rangle_{\Gamma_\partial}&=&0\quad\forall\mu\in M_h,\label{hdg41}\\
  \widehat{\bs}_\bq(0)\cdot \bn=\bs_\bq(0)\cdot\bn&-&\tau(\bs_u(0)-\widehat{\bs}_u(0))\;\;\;\mbox{on}\;\mathcal{E}_h.\label{erflux}
\end{eqnarray}
\end{subequations}
Now, a use of  arguments in \cite{cockburn2010projection} (see also \cite[Lemma 3.6]{cockburn2018stormer}, completes 
estimate \eqref{estimate:IC}. In order to prove \eqref{estimate:IC-0}, we appeal to the Aubin-Nitsche duality argument as in 
\cite[Lemma 4.1, Theorem 4.1]{cockburn2010projection} to conclude the remaining part of the proof.
\end{proof}

By using the semidiscrete formulation \eqref{ereq}, below, we estimate $\|\bs_{u_t}\|$ and $\|\bs_{\bq}\|.$  

\begin{lemma}\label{ll3} 
With $u_{h0}$ defined by HDG elliptic problem \eqref{HDG0}, $u_{h1}= \pw u_1,$ and $u\in W^{2,1}(0,T;H^{k+2}),$  there exists a positive constant $C$ independent of $h$ such that
\begin{equation}
 \begin{split}
\Big(\|\bs_{\bq}\|^2+\|\theta_{u_{t}}\|^2+\|\bs_u-\widehat{\bs}_u\|_+^2\Big)^{1/2}\leqslant C\;h^{k+1} \Big( \|u_0\|_{H^{k+2}}
+ \|u\|_{W^{2,1}(0,T;H^{k+2})} \Big)\label{ll31}. \end{split}
\end{equation}

Moreover, there holds
 for  $p\in [2,\infty)$
 \begin{equation}\label{estimate:p0}
 \|\theta_{u}(t)\|_{L^p} \leqslant C\;h^{k+1} \Big( \|u_0\|_{H^{k+2}}
+ \|u\|_{W^{2,1}(0,T;H^{k+2})} \Big), 
 \end{equation}
 and  for $p=\infty$
  \begin{equation}\label{estimate:infty0}
 \|\theta_{u}(t)\|_{L^{\infty} }\leqslant C\;\log \frac{1}{h}\;h^{k+1} \Big( \|u_0\|_{H^{k+2}}
+ \|u\|_{W^{2,1}(0,T;H^{k+2})} \Big).
 \end{equation}

\end{lemma} 

\begin{proof}
 For estimating $\|\theta_{u_t}\|$, differentiate 
 \eqref{ee1} with respect to $t$ and then, choose $\bv_h:=\bs_{\bq}$ and $w_h:=\theta_{u_t}$ in \eqref{ee2} to obtain
 \begin{equation}
 \begin{split}
 \frac{1}{2}\frac{d}{dt}\left(\|\bs_{\bq}\|^2+\|\theta_{u_{t}}\|^2+\|\bs_u-\widehat{\bs}_u\|_+^2\right)+(f(u)-f(u_h),\theta_{u_{t}})=(\br_{\bq_{t}},\bs_{\bq})+(\br_{u_{tt}},\theta_{u_{t}}). \label{th6}
 \end{split}
 \end{equation}
 With the help of the Cauchy-Schwarz inequality, it follows that
 \begin{align*}
      \frac{d}{dt}\left(\|\bs_{\bq}\|^2+\|\theta_{u_{t}}\|^2+\|\bs_u-\widehat{\bs}_u\|_+^2\right)&\leqslant \Big(\|f(u)-f(u_h)\| +\|\br_{\bq_{t}}\|+\|\br_{u_{tt}}\|\Big)\Big(\|\bs_{\bq}\|^2+\|\theta_{u_{t}}\|^2+\|\bs_u-\widehat{\bs}_u\|_+^2\Big)^{1/2}, \label{th7}
 \end{align*}
and  hence,
 \begin{equation}
 \begin{split}
 \frac{d}{dt}\left(\|\bs_{\bq}\|^2+\|\theta_{u_{t}}\|^2+\|\bs_u-\widehat{\bs}_u\|_+^2\right)^{1/2} \leqslant \|f(u)-f(u_h)\| +\|\br_{\bq_{t}}\|+\|\br_{u_{tt}}\|. \label{th8}
 \end{split}
 \end{equation}
With the help of HDG projection, we split the nonlinear terms into two parts.  Since  $f(u)=u^3-u$, using the mean value theorem, for some $\lambda\in (0,1)$ and the Cauchy-Schwarz inequality
\begin{eqnarray}
  \| f(u)-f(\pw u)\| &\leqslant& C\|1+(\pw u)^2+u^2\|_{L^\infty}\|\br_u\|\nonumber\\
  &\leqslant& C\;\|\br_u\|. \label{th9}
  \end{eqnarray}
Here, we have used boundedness of $\pw u \in L^{\infty}(\Omega) $ as $u\in L^{\infty}(\Omega)$ is bounded and  
using an inverse estimate and Theorem-\ref{tt1}. 
Now, a use of Holder's inequality with Lemma \ref{ll2} shows
 \begin{align*}
  \|f(\pw u)-f(u_h)\|^2 &\leqslant C\int_{\Omega}(1+|\pw u|^4+|u_h|^4)|\theta_u|^2ds\\
                     &\leqslant C\;\big((1+\|\pw u\|_{L^{6}}^4+\|u_h\|_{L^{6}}^4)\|\theta_u\|_{L^{6}}^2\big)\\
                &\leqslant C\;\big((1+\|\bq_h\|^4 + \|u_h-\widehat{u}_h\|_+^4)(\|\theta_{\bq}\|^2 + \|\theta_h-\widehat{\theta}_h\|_++\|\br_{\bq}\|^2)\big).
  \end{align*}
 Now, a direct consequence of Corollary 2.1 shows that $1+\|\bq_h\|^4 + \|u_h-\widehat{u}_h\|_+^4\leqslant C.$ 
 Hence,
 \begin{align}
 \|f(\pw u)-f(u_h)\|^2 &\leqslant C\;(\|\theta_{\bq}\|^2 + \|\theta_u-\widehat{\theta}_u\|_+^2+\|\br_{\bq}\|^2). \label{th10}
 \end{align}
 On combining \eqref{th9} and \eqref{th10}, there holds
\begin{align}
 \|f(u)-f(u_h)\|^2 &\leqslant C\;(\|\br_u\|^2+\|\theta_{\bq}\|^2 + \|\theta_u-\widehat{\theta}_u\|_+^2+\|\br_{\bq}\|^2). \label{th131}
 \end{align} 
 A use of inequality \eqref{th131} in \eqref{th8} yields
  \begin{align*}
  \frac{d}{dt}\left(\|\bs_{\bq}\|^2+\|\theta_{u_{t}}\|^2+\|\bs_u-\widehat{\bs}_u\|_+^2\right)^{1/2}&\leqslant C\;\big(\|\br_u\|+\|\br_{\bq}\| +\|\br_{\bq_{t}}\|+\|\br_{u_{tt}}\|\\&+\|\theta_{\bq}\| + \|\theta_u-\widehat{\theta}_u\|_+\big).
 \end{align*}
 After integrating the above inequality and noting that $\theta_{u_{t}}(0)=0,$ we use Gr\"{o}nwall's lemma to arrive at
  \begin{eqnarray}\label{1111}
 && \left(\|\bs_{\bq}\|^2+\|\theta_{u_{t}}\|^2+\|\bs_u-\widehat{\bs}_u\|_+^2\right)^{1/2} \leqslant  \Big( \left(\|\bs_{\bq}(0)\|^2+\|(\bs_u-\widehat{\bs}_u)(0)\|_+^2\right)^{1/2} \nonumber\\
 &&\;\;\;\;\;\;\;\;\;\;\;\;\;\;\;+  \int_{0}^{t} \big( \|\br_u\|+ \|\br_{\bq}\| +\|\br_{\bq_{t}}\|+\|\br_{u_{tt}}\|\big)\;ds\Big) \exp{(Ct)}.
 \end{eqnarray}
 A use of  Lemma~\ref{tt1} with Lemma~\ref{IC-approx-1} completes the  estimate \eqref{ll31}.\\
  A use of Lemma \ref{ll2} for the  error equation \eqref{ee1} shows that  for  $p\in [2,\infty)$
 \begin{equation}\label{estimate:p}
 \|\theta_{u}(t)\|_{L^p} \leqslant C\;\left(\|\bs_{\bq}\|+\|\bs_u-\widehat{\bs}_u\|_{+} +\|\br_{\bq}\| \right),
 \end{equation}
 and  for $p=\infty$
  \begin{equation}\label{estimate:infty}
 \|\theta_{u}(t)\|_{L^{\infty} }\leqslant C\;\log \frac{1}{h}\;\left(\|\bs_{\bq}\|+\|\bs_u-\widehat{\bs}_u\|_{+}+\|\br_{\bq}\| \right).
 \end{equation}
 Then, a use of \eqref{ll31} completes the rest of the proof.
 \end{proof}
 \begin{lemma}\label{lem34} Under the assumptions in Lemma~\ref{ll3} and with additional assumption $u_{ttt}\in L^1(0,T;H^{k+2}),$
there holds for a positive constant $C$ independent of $h$: 
\begin{eqnarray}
\Big(\|\bs_{\bq_t}\|^2 +\|\theta_{u_{tt}}\|^2+\|\bs_{u_t}-\widehat{\bs}_{u_t}\|_+^2\Big)^{1/2} \leqslant 
C\;h^{k+1} \Big( \|u_0\|_{H^{k+2}} + \|u_1\|_{H^{k+2}} + \|u\|_{W^{3,1}(0,T;H^{k+2})} \Big). \label{ll32}
\end{eqnarray}
\end{lemma} 
\begin{proof}
  In order to prove \eqref{ll32}, now differentiate \eqref{ee1} twice with respect to time and differentiate \eqref{ee2}-\eqref{ee5} and then choose $\bv_h:=\bs_{\bq_t}$ and $w_h:=\theta_{u_{tt}}$ to obtain
  \begin{equation*}
 \begin{split}
 \frac{1}{2}\frac{d}{dt}\left(\|\bs_{\bq_t}\|^2+\|\theta_{u_{tt}}\|^2+\|\bs_{u_t}-\widehat{\bs}_{u_t}\|_+^2\right)+(f_{t}(u)-f_{t}(u_h),\theta_{u_{tt}})&=(\br_{\bq_{tt}},\bs_{\bq_t})+(\br_{u_{ttt}},\theta_{u_{tt}}). 
 \end{split}
 \end{equation*}
 A use of  the Cauchy-Schwarz inequality shows
 \begin{equation*}
 \begin{split}
 \frac{d}{dt}\left(\|\bs_{\bq_t}\|^2+\|\theta_{u_{tt}}\|^2+\|\bs_{u_t}-\widehat{\bs}_{u_t}\|_+^2\right)^{1/2}\leqslant \Big(\|f_{t}(u)-f_{t}(u_h)\|+\|(\br_{\bq_{tt}})\|+\|(\br_{u_{ttt}})\|\Big)  
 \end{split}
 \end{equation*}
 Integrating 
 from $0$ to $t$ for $t\in (0,T]$, we arrive at
\begin{eqnarray}
\Big(\|\bs_{\bq_t}\|^2+\|\theta_{u_{tt}}\|^2+\|\bs_{u_t}-\widehat{\bs}_{u_t}\|_+^2\Big)^{1/2}&\leqslant&
\Big(\|\bs_{\bq_t}(0)\|^2+\|\theta_{u_{tt}}(0)\|^2+\|(\bs_{u_t}-\widehat{\bs}_{u_t})(0)\|_+^2\Big)^{1/2}\nonumber\\
&+& C \int_0^t\big(\|\br_{\bq_{tt}}\|+\|\br_{u_{ttt}}\| +\|f_{t}(u)-f_{t}(u_h)\|\big)ds. \label{th11}
\end{eqnarray}
Since $f(u)=u^3-u$, there holds
 \begin{align}
 f_{t}(u)-f_{t}(\pw u) 
 &=3(u^2u_t-(\pw u)^2 \pw u_t)-(u_t-\pw u_t)\nonumber\\
 &=(3u^2-1)(u_t-\pw u_t)+ 3\pw u_t(u^2-(\pw u)^2)\nonumber
 \end{align}
and using the boundedness of  $\pw u$ and $\pw u_t$ in $L^{\infty}(\Omega),$ it follows that
 \begin{align}
 \|f_{t}(u)-f_{t}(\pw u) &\|\leqslant  \|3u^2+1\|_{L^{\infty}}\|u_t-\pw u_t\|+\|3\pw u_t\|_{L^{\infty}}\|u+\pw u\|_{L^{\infty}}\|u-\pw u\|\nonumber\\
&\leqslant  C \Big(\|\br_{u_{t}}\| +\|\br_u\|\Big). \label{1112}
 \end{align}
 Similarly, we arrive  with replacing $u_{ht}$ by $\bs_{u_{t}} + \pw u_t$ at
\begin{align}
\|f_{t}(\pw u)-f_{t}(u_h)\|&\leqslant \|(3(\pw u)^2-1) \bs_{u_{t}}\|+3\| (\bs_{u_{t}} + \pw u_t) (\pw u+u_{h}) \bs_u\| \nonumber\\
&\leqslant \|3(\pw u)^2+1\|_{L^{\infty}}\|\bs_{u_{t}}\|+3\|\pw u+u_{h}\|_{L^{\infty}} \left(\|\bs_{u_{t}}\|\; \|\bs_u\|_{L^{\infty}} +
\|\pw u_t\|_{L^{\infty}} \|\bs_u\|\right) \nonumber\\
&\leqslant C\;\big( (1+\|\bs_u\|_{L^{\infty}})\;\|\bs_{u_{t}}\| +\|\bs_u\|\big).\label{1112-1}
\end{align}

Therefore, adding \eqref{1112} with\eqref{1112-1}, it follows using  Lemma \ref{ll2} and inequality \eqref{estimate:infty} that
\begin{align}\label{1113}
    \| f_t(u)-f_t(u_h)\| & \leqslant C \Big( 1+ \log \frac{1}{h}\;\left(\|\bs_{\bq}\|+\|\bs_u-\widehat{\bs}_u\|_{+}+\|\br_{\bq}\| \right)\Big)\;\int_{0}^t \left(\|\bs_{\bq}\|+\|\bs_u-\widehat{\bs}_u\|_{+}+\|\br_{\bq}\| \right)\;ds \nonumber\\
  &  +C \left(\|\bs_{\bq}\|+\|\bs_u-\widehat{\bs}_u\|_{+}+\|\br_{\bq}\| \right)+\|\br_{u_{t}}\| +\|\br_u\|.
\end{align}
Now, it remains  to estimate $\|\bs_{u_{tt}}(0)\|.$ We rewrite \eqref{ee2} at $t=0$ using  \eqref{hdg21} to obtain
\begin{equation}\label{aa}
    (\bs_{u_{tt}}(0),w_h)+(f(u)(0)-f(u_h(0)),w_h)=(\br_{u_{tt}}(0),w_h),
\end{equation}
Note that
\begin{align}
 \|(f(u)-f(u_h))(0)\|^2 &\leqslant C\;(\|\br_u(0)\|^2+\|\br_{\bq}(0)\|^2). \label{th13}
 \end{align} 
Finally, set $w_h = \bs_{u_{tt}}(0)$ in \eqref{aa}, and use the bound \eqref{th13} to arrive at
\begin{equation}\label{ab}
    \|\bs_{u_{tt}}(0)\|\leqslant C (\|\br_u(0)\|+\|\br_{\bq}(0)\|)+\|\br_{u_{tt}}(0)\|,
\end{equation}
Similarly, for $\bs_{\bq_t}(0)$, differentiating the first error equation with respect to time then write the resulting equation at $t=0$ and set $\bv_h=\bs_{\bq_t}(0)$, to obtain
\begin{equation}\label{ac}
    (\bs_{\bq_t}(0),\bs_{\bq_t}(0))+(\bs_{u_t}(0),\nabla\cdot\bs_{\bq_t}(0))-\langle \widehat{\bs}_{u_t}(0),\bs_{\bq_t}(0)\cdot\bn\rangle=(\br_{\bq_t}(0),\bs_{\bq_t}(0)).
\end{equation}
Here $(\bs_{u_t}(0),\nabla\cdot\bs_{\bq_t}(0))$ is zero, as we choose $u_{ht}(0)=\pw u_1$.
Now, a  use of  \eqref{hdg31} with \eqref{hdg41} and \eqref{erflux} at $t=0$ in \eqref{ac} shows after simplification
\begin{equation}\label{qt-0}
    \|\bs_{\bq_t}(0)\|^2+\|(\bs_{u_t}-\widehat{\bs}_{u_t})(0)\|_+^2 \leqslant C \|\br_{\bq_{t}}(0)\|^2.
\end{equation}
On combining  \eqref{1112}, \eqref{1113}, \eqref{ab} and \eqref{qt-0} in \eqref{th11}, an application of Lemmas~\ref{tt1} and \ref{ll3} completes the rest of the proof.

 \end{proof}
 
 \begin{theorem}\label{tt2} 
 Let $u \in L^{\infty}(0,T;H^{k+2})$, $u_t \in L^{\infty}(0,T;H^{k+2})$ 
 and $u_{tt} \in L^{1}(0,T;H^{k+2}).$
 With assumptions in Lemma~\ref{ll3}, the following estimates holds true for all $h= \max\limits_{K\in \mathscr{T}_h}h_K$ and for all $t \in (0,T]$:
\begin{equation}\label{321}
    \|(u_t-u_{ht})(t)\| + \|(\bq-\bq_h)(t)\|\leqslant {C\;h^{k+1} \Big( \|u_0\|_{H^{k+2}} + \|u\|_{W^{2,1}(0,T;H^{k+2})} + \|u\|_{W^{1,2}(0,T;H^{k+2})} \Big)}.
\end{equation}
If in addition $u_{tt}\in L^{\infty}(0,T;  H^{k+2})$  and $u_{ttt}\in L^{2}(0,T;H^{k+2}(\Omega)),$ then
\begin{align}\label{322}
     \|(u_{tt} - u_{htt})(t)\| + \|(\bq_t - \bq_{ht})(t)\| &\leqslant {C\;h^{k+1} \Big( \|u_0\|_{H^{k+2}} + \|u_1\|_{H^{k+2}} + \|u\|_{W^{3,1}(0,T; H^{k+2})} + \|u\|_{W^{2,2}(0,T;H^{k+2})} \Big)}.
\end{align}
\end{theorem}
 \begin{proof}
 With the help of HDG projections and triangle inequality, we write \[\|(\bq-\bq_h)(t)\|\leqslant \|(\bq-\pV \bq)(t)\|+\|(\pV \bq-\bq_h)(t)\|.\] 
  Now, from Theorem-\ref{tt1} and Lemma-\ref{ll3}, the result follows, completing the rest of the proof.
 \end{proof}
 
 As a consequence of \eqref{321} and using the fact that $\phi(t) = \phi(0) + \int_{0}^t \phi(s)\;ds$, we obtain 
 \begin{equation}\label{L2-estimate-0}
  \|(u-u_h)(t)\| \leqslant {C\;h^{k+1} \Big( \|u_0\|_{H^{k+2}} + \|u\|_{W^{2,1}(0,T;H^{k+2})} + \|u\|_{W^{1,2}(0,T;H^{k+2})} \Big)},
 \end{equation}
 under higher regularity : $u \in L^{\infty}(0,T; H^{k+2} )$, $u_t \in L^{\infty}(0,T;H^{k+2} )$ and $u_{tt} \in L^{2}(0,T;H^{k+2}).$
 
 Below, we discuss using a variant of Baker's nonstandard energy technique, see \cite{baker1976error} and \cite{sinha1998effect} a direct proof of the optimal order of convergence in $L^{\infty}(0,T;L^2)$ norm with reduced regularity results on the exact solution.
 
  For any $\widetilde{\phi}$, define $\widetilde{\phi}(s)=\int_0^{s}\phi(z)dz$ for $0\leqslant s\leqslant t$. 
 
 \begin{lemma}\label{ll4} With $u_{h0}$ defined by HDG elliptic problem \eqref{HDG0}, $u_{h1}= \pw u_1,$ 
  let $u  \in W^{1,2} (0,T; H^{k+2} ).$ 
 Then, here exists a positive constant $C$ independent of $h$ such that
   \begin{align*}
\Big(\|\widetilde{\bs}_\bq\|^2+\|\theta_{u}\|^2+\|\widetilde{\theta}_u-\widetilde{\widehat{\theta}}_u\|_+^2\Big)^{1/2}
\leqslant  C\;h^{k+1}\; \Big( \|u_0\|_{H^{k+2}} + \|u\|_{W^{1,2}(0,T; H^{k+2})}\Big).
\end{align*}
\end{lemma}
 \begin{proof}
Integrating \eqref{ee2} from $0$ to $s$ and choosing $w_h:=\theta_{u}$ and $\bv_h:=\widetilde{\bs}_{\bq}$ in \eqref{ee1} and adding, we obtain 
\begin{align}
(\theta_{u_s},\theta_{u}) + (\bs_\bq,\widetilde{\bs}_{\bq})+\langle\tau(\widetilde{\theta}_u-\widetilde{\widehat\theta}_u),
\theta_u-\widehat\theta_u\rangle=(\br_{\bq},\widetilde{\bs}_{\bq})+(\br_{u_{s}},\theta_{u})-\Big(\int_0^{s}(f(u)-f(u_h))dz,\theta_u\Big).\label{th0}
 \end{align}
 Now, rewrite \eqref{th0} as
 \begin{align}
 \frac{1}{2}\frac{d}{ds}\left(\|\widetilde{\bs}_\bq\|^2+\|\theta_{u}\|^2+\|\widetilde{\theta}_u-\widetilde{\widehat{\theta}}_u\|_+^2\right)=(\br_{\bq},\widetilde{\bs}_{\bq})+(\br_{u_{s}},\theta_{u})-\Big(\int_0^{s}(f(u)-f(u_h))dz,\theta_{u}\Big).\label{th1}
 \end{align}
A use of the Cauchy-Schwarz and Young's inequalities in \eqref{th1}, we arrive at
 \begin{align}
     \frac{1}{2}\frac{d}{ds} \left(\|\widetilde{\bs}_\bq\|^2+\|\theta_{u}\|^2+\|\widetilde{\theta}_u-\widetilde{\widehat{\theta}}_u\|_+^2\right)&\leqslant C \Big(\|\br_{\bq}\|^2+\|\widetilde{\bs}_{\bq}\|^2+\|\br_{u_{s}}\|^2+\|\theta_{u}\|^2\nonumber\\&+ \left\Vert\int_0^{s}(f(u)-f( u_h))dz\right\Vert^2\Big).\label{3.21}
 \end{align}
 Now, with the help of HDG projection and equation \eqref{th9}, we obtain 
 \begin{align}
   \left\Vert\int_0^{s}(f(u)-f(\pw u))dz\right\Vert^2 
    &\leqslant C\int_0^{s}(\|1+(\pw u)^2+u^2\|^2{_{L^\infty}}\|\br_u\|^2)dz.\label{3.20}
\end{align}
Also, using the mean value theorem, for some $\lambda \in (0,1)$,
\begin{align*}
    f(\pw u)-f(u_h) &= \big(3((1-\lambda)\pw u+\lambda u_h)^2-1\big)(u_h-\pw u),\nonumber\\
    &= (3((1-\lambda)\pw u+\lambda u_h)^2-1)\dfrac{d}{ds}\widetilde{\theta}_u,\nonumber\\
    &= \dfrac{d}{ds}\big((3((1-\lambda)\pw u+\lambda u_h)^2-1)\widetilde{\theta}_u\big)\nonumber\\&-6((1-\lambda)\pw u+\lambda u_h)((1-\lambda)\pw u_s+\lambda u_{hs})\widetilde{\theta}_u.
\end{align*}
Hence, 
 \begin{align*}
      \left\Vert\int_0^{s}\big(f(\pw u)-f(u_h)\big)dz\right\Vert^2
    &\leqslant C\;\big((1+\|\pw u\|^{4}_{L^6}+\|u_{h}\|^{4}_{L^6}\|)\|\widetilde{\theta}_u\|^{2}_{L^6}\nonumber\\&+\int_0^{s}(1+\|\pw u\|^{2}_{L^6}+\|u_{h}\|^{2}_{L^6})(1+\|\pw u_s\|^{2}_{L^6}+\|u_{hs}\|^{2}_{L^6})\|\widetilde{\theta}_u\|^2_{L^6}dz\big).
 \end{align*}
If we consider the functions $u,\ u_s$
$\in L^\infty(\Omega)\subset L^6(\Omega)$ 
 and apply the $L^p$ estimates from Lemma-\ref{ll2} and use the second bound, derived from Theorem-\ref{tt2}, we obtain
 \begin{align}
      \left\Vert\int_0^{s}(f(\pw u)-f(u_h))dz\right\Vert^2
    &\leqslant C\;\big(\|\widetilde{\bs}_\bq\|^2+\|\widetilde{\br}_{u}\|^2+\|\widetilde{\theta}_u-\widetilde{\widehat{\theta}}_u\|_+^2\nonumber\\&+\int_0^{s}(\|\widetilde{\bs}_\bq\|^2+\|\widetilde{\br}_{u}\|^2+\|\widetilde{\theta}_u-\widetilde{\widehat{\theta}}_u\|_+^2)\;dz\big).\label{3.22}
 \end{align}
 Combine expression \eqref{3.20} and \eqref{3.22}, and then put in the \eqref{3.21} to arrive at 
\begin{align}
\frac{d}{ds}\left(\|\widetilde{\bs}_\bq\|^2+\|\theta_{u}\|^2+\|\widetilde{\theta}_u-\widetilde{\widehat{\theta}}_u\|_+^2\right) &\leqslant C \Big(\|\br_{\bq}\|^2+\|\br_{u_{s}}\|^2+\|\theta_{u}\|^2+\|\br_u\|^2+\|\widetilde{\bs}_{\bq}\|^2\nonumber\\&+\|\widetilde{\br}_{u}\|^2+\|\widetilde{\theta}_u-\widetilde{\widehat{\theta}}_u\|_+^2+\int_0^{s}(\|\widetilde{\bs}_\bq\|^2+\|\widetilde{\br}_{u}\|^2\nonumber\\&+\|\widetilde{\theta}_u-\widetilde{\widehat{\theta}}_u\|_+^2)\;dz\Big).\label{3.23}
 \end{align}
Integration of the above-resulting equation \eqref{3.23} yields
\begin{align}
\|\widetilde{\bs}_\bq\|^2&+\|\theta_{u}\|^2+\|\widetilde{\theta}_u-\widetilde{\widehat{\theta}}_u\|_+^2\leqslant C \Big(\|\theta_{u}(0)\|^2+\int_0^{t}\Big(\|\br_{\bq}\|^2+\|\br_{u_{s}}\|^2+\|\theta_{u}\|^2+\|\br_u\|^2\nonumber\\
&+\|\widetilde{\bs}_{\bq}\|^2+\|\widetilde{\br}_{u}\|^2+\|\widetilde{\theta}_u-\widetilde{\widehat{\theta}}_u\|_+^2\Big)\;ds+\int_0^{t}\int_0^{s}(\|\widetilde{\bs}_\bq\|^2+\|\widetilde{\br}_{u}\|^2
+\|\widetilde{\theta}_u-\widetilde{\widehat{\theta}}_u\|_+^2)\;dzds.\label{3.24}
 \end{align}
Solving the double integral term in \eqref{3.24}, we can write it,
\begin{align*}
\|\widetilde{\bs}_\bq\|^2&+\|\theta_{u}\|^2+\|\widetilde{\theta}_u-\widetilde{\widehat{\theta}}_u\|_+^2\leqslant 
C\; (1+t) \;\Big( \|\theta_{u}(0)\|^2+\int_0^{t}\Big(\|\br_{\bq}\|^2+\|\br_{u_{s}}\|^2+\|\br_u\|^2\Big)\;ds\\
&+ C\;(1+t)\;\int_{0}^{t} \Big(\|\widetilde{\bs}_{\bq}\|^2+\|\theta_{u}\|^2+\|\widetilde{\theta}_u-\widetilde{\widehat{\theta}}_u\|_+^2\Big)\;ds.
 \end{align*}
Finally, an application of the Gr\"{o}nwall's lemma with projection estimates from Lemma~\ref{tt1} completes the rest of the proof. 
 \end{proof}

  \begin{theorem}\label{tt3} With $u_{h0}$ defined by HDG elliptic problem \eqref{HDG0}, $u_{h1}= \pw u_1,$
 let $u \in L^{\infty}(0,T; H^{k+2} )$ and $u_t \in L^{2}(0,T; H^{k+1} ).$ 
 Then, there exists a positive constant $C$ independent of $h$ such that  for all $h= \max\limits_{K\in \mathscr{T}_h}h_K$ and for all $t \in (0,T]$, the following estimate holds:
\begin{equation*}
    \|(u-u_{h})(t)\| + \|(\widetilde{\bq}-\widetilde{\bq}_h)(t)\|\leqslant {C\;h^{k+1}\; \Big( \|u_0\|_{H^{k+2}} + \|u\|_{L^{\infty}(0,T; H^{k+2})}+ \|u_t\|_{L^2(0,T; H^{k+2})}\Big)}.
\end{equation*}
\begin{proof}
   With the help of HDG projections, we can write $\|u-u_h\|\leqslant \|u-\pw u\|+\|\pw u-u_h\|.$ \\
Now from Theorem-\ref{tt1} and Lemma-\ref{ll4}, the result follows, and this completes the proof. 
\end{proof}
\end{theorem}


\section{Post-processing}
This section deals with local post-processing of the semi-discrete approximation of the displacement variable, which provides a super-convergence result when compared with the exact solution.


For HDG approaches, Cockburn {\it et al.} devised element-by-element post-processing of the scalar variable for elliptic problems and the velocity variable in the Stokes problem in \cite{cockburn2012conditions,cockburn2013conditions}. He also demonstrated how to use this local postprocessing to achieve a close approximation to the original scalar unknown that also converges with order $O(h^{k+2})$ for $k \geqslant 1$ for the acoustic wave problem \cite{cockburn2014uniform}. Inspired by their approach, we seek $u^*_h$ in the space
\begin{align*}
W_h&=\{w\in L^2(\Omega):w|_K\in \mathcal{P}_{k+1}(K)~\forall ~K\in \mathscr{T}_h\},
\end{align*}
such that, in each element $K\in \mathscr{T}_h$,  it satisfies:
\begin{equation}\left.
\begin{array}{rl}
(\nabla u^*_h, \nabla w) &=(\bq_h, \nabla w)  \quad\forall w\in \mathcal{P}_{k+1}(K),\\
(u^*_h,1) &= (u_h,1).
\end{array}\right\}\label{main1}
\end{equation}
\subsection{The BDM Projection}
 For our subsequent analysis, we define the BDM projection, introduced earlier by Brezzi, Douglas, and Marini \cite{brezzi1985two} and denoted by ${\bf{\Pi}}^{BDM}$ is an operator that maps a vector-valued function $\bi\in \bH^1(\Omega)$ onto  BDM finite dimensional space $\bV_h$ 
 satisfying:
\begin{subequations}\label{bdm1}
\begin{eqnarray}
     ({\bf{\Pi}}^{BDM}\bi , v )_K  &=& ( \bi , v )_K, \quad \forall  v \in [\mathcal{P}_k(K)]^d, \\
     \langle{\bf{\Pi}}^{BDM}\bi , p\rangle_E &=& \langle\bi ,p \rangle_E, \quad \forall  p \in \mathcal{P}_k(E), \quad \forall E \in \partial K
\end{eqnarray}
\end{subequations}
with property,
\begin{align}\label{bdm} 
(\nabla\cdot{\bf{\Pi}}^{BDM}\bi, v)_K = (\nabla\cdot\bi, v)_K \quad \forall v\in\mathcal{P}_k(K),
\end{align} 
and
\begin{align}\label{BDM-projection}
    \|{\bf{\Pi}}^{BDM}\bi-\bi\| \leqslant & C\;h\|\bi\|_{H^1(K)}  \quad \forall \bi\in H^1(K).  
\end{align} 
The main theorem of this section is given below.
\begin{theorem}\label{th-pp}
Under the assumptions of Theorem~\ref{tt2}, there holds for $k\geqslant1$  
\begin{align*}
\|u-u_h^{*}\| \leqslant { C\;h^{k+2}\Big( {\|u_0\|_{H^{k+2}}}+\|u_1\|_{H^{k+2}} + \|u\|_{W^{3,1}(0,T;H^{k+2})}+\|u\|_{W^{0,2}(0,T;H^{k+2})} \Big)}.
\end{align*}
\end{theorem}
To prove this theorem, we first prove a superconvergence of $P_{k-1}\theta_u$ with the help of the duality argument.
\subsection{The dual problem} 
Now, we introduce a dual problem of the wave equation as follows. We define the function $\phi \in H^1_0(\Omega)$ to be the solution of the wave equation and for any given function $\Theta \in L^2(\Omega)$, $\phi$ satisfying
\begin{subequations}\label{ereq3}
\begin{eqnarray}
{\phi}_{ss}(s)-\Delta\phi(s) + f_u(u)\phi(s) &=& 0,~~~~~~~~~\text{in }\Omega\times[0,t),\label{ep1}\\
\phi &=&0,\label{ep3}~~~~~~~~~\text{on }\partial\Omega\times[0,t),\\
\phi(t)&=&0, ~~~~~~~~~\text{in }\Omega,\label{ep4}\\
\phi_s(t)&=&\Theta, ~~~~~~~~~\text{in }\Omega.\label{ep5}
\end{eqnarray}
\end{subequations}
\textbf{Proposition 4.1}. With $u \in {L^\infty(\Omega)}$, there is a positive constant $C$ such that  the following holds:
\begin{align}
\|\phi\|_{L^\infty(0,t;L^2(\Omega))}+\|\phi\|_{L^\infty(0,t;H^1(\Omega))}+\|\phi_t\|_{L^\infty(0,t;L^2(\Omega))}\leqslant C \|\Theta\|.\label{p0} 
\end{align}
Moreover, 
\begin{align}
    \|\undertilde{\phi}\| _{L^\infty(0,t;H^2(\Omega))}\leqslant C \|\Theta\|,\label{p2}
\end{align}
where $\|\undertilde{\phi}(s)\| = \int_s^t \phi(z)\;dz$.
\begin{proof}
By change of variable $s\longrightarrow t-s$ and with $\bar{\phi}(s)= \phi(t-s),$ the problem \eqref{ereq3} in  transformed into a forward problem in $\bar{\phi}.$

\begin{subequations}\label{ereq3-1}
\begin{eqnarray}
\bar{\phi}_{ss}(s)-\Delta \bar{\phi}(s) + f_u(\bar{u}) \bar{\phi}(s) &=& 0,~~~~~~~~~\text{in }\Omega\times[0,t),\label{ep1-1}\\
\bar{\phi} &=&0,\label{ep3-1}~~~~~~~~~\text{on }\partial\Omega\times[0,t),\\
\bar{\phi}(0)&=&0, ~~~~~~~~~\text{in }\Omega,\label{ep4-1}\\
\bar{\phi}_s(0)&=&-\Theta, ~~~~~~~~~\text{in }\Omega.\label{ep5-1}
\end{eqnarray}
\end{subequations}
Now, the standard energy argument along with the boundedness of $f_u$ and Gr\"{o}nwall's inequality, shows
\begin{align}
\|\bar{\phi}\|_{L^\infty(0,t;L^2(\Omega))}+\|\bar{\phi}\|_{L^\infty(0,t;H^1(\Omega))}+\|\bar{\phi}_t\|_{L^\infty(0,t;L^2(\Omega))}\leqslant C\; \|\Theta\|.\label{p0-1} 
\end{align}
and  on integrating concerning time, use of \eqref{p0-1} yields in a straightforward way the estimate
\begin{align}
    \|\undertilde{\bar{\phi}}\| _{L^\infty(0,t;H^2(\Omega))}\leqslant C \|\Theta\|,
\end{align}
where $\|\undertilde{\bar{\phi}}(s)\| = \int_s^t \bar{\phi}(z)\;dz.$
Then an application of inverse transformation $t-s \longrightarrow s$ and with $\bar{\phi}(t-s)= \phi(s)$ completes the rest of the proof.
\end{proof}

Next, in  order  to prove the Theorem \ref{th-pp}, we need to find the estimate of $P_{k-1}\bs_u$
\subsection{Estimate of \texorpdfstring{$P_{k-1}\bs_u$}{TEXT}}

\begin{lemma}\label{lemma-pp}
Under the assumptions of Theorem~\ref{tt2}, there holds for $k\geqslant1$  
\begin{align}
   \|P_{k-1}\bs_u(t)\| \leqslant {C\;h^{k+2}\Big( \|u_0\|_{H^{k+2}}+ \|u_1\|_{H^{k+2}} + \|u\|_{W^{3,1}(0,T;H^{k+2})} \Big)}.  \label{pp3}
\end{align}
\end{lemma}
\begin{proof}
On multiplying $P_{k-1}\theta_u$ to \eqref{ep1} and integrating over $\Omega$,
we note that
\begin{align*}
\int_0^t \frac{d}{ds}\Big((\phi_s,P_{k-1}\theta_u)-(\phi,P_{k-1}\theta_{u_s})\Big) &= \int_0^t\Big(-(\phi,P_{k-1}\theta_{u_{ss}}) + (\Delta\phi,P_{k-1}\theta_u)-\Big( f_u(u)\phi,P_{k-1}\theta_u)\Big)\;ds.
\end{align*}
Using equation \eqref{ep5} from the dual problem, the property of projection and the property of BDM-projection \eqref{bdm}, gives
\begin{align*}
(\Theta,P_{k-1}\theta_u)&= (\phi_s(0),P_{k-1}\theta_u(0))-(\phi(0),P_{k-1}\theta_{u_s}(0))+\int_0^t\Big(-(\phi,P_{k-1}\theta_{u_{ss}}) \\&+ \big(\nabla\cdot ({\bf{\Pi}}^{BDM}\nabla\phi),\theta_u\big)-( f_u(u)\phi,P_{k-1}\theta_u)\Big)\;ds.
\end{align*}
Since, $u_{ht}(0)=\pV u_1,$  the second term on the right-hand side become zero.
For the second last term under the time integral, use of  \eqref{ee1} shows
$$\big(\nabla\cdot ({\bf{\Pi}}^{BDM}\nabla\phi),\theta_u\big) = (\pV\bq-\bq,{\bf{\Pi}}^{BDM}\nabla\phi)-(\bs_\bq,{\bf{\Pi}}^{BDM}\nabla\phi)+\langle \widehat{\theta}_u,{\bf{\Pi}}^{BDM}\nabla\phi\cdot\bn\rangle.$$
Here, the boundary term becomes zero as $\widehat\theta_u = 0 $ on $\partial\Omega$ and ${\bf{\Pi}}^{BDM}\nabla\phi \in H(div,\Omega)$. Thus, we arrive at
\begin{align*}
(\Theta,P_{k-1}\theta_u)
&= (\phi_s(0),P_{k-1}\theta_u(0))+\int_0^t\Big(-(\phi,P_{k-1}\theta_{u_{ss}})+(\bq_h-\bq,{\bf{\Pi}}^{BDM}\nabla\phi)-( f_u(u)\phi,P_{k-1}\theta_u)\Big)ds,\\
&= (\phi_s(0),P_{k-1}\theta_u(0))+\int_0^t\Big(-(P_{k-1}\phi,\theta_{u_{ss}})+(\bq_h-\bq,{\bf{\Pi}}^{BDM}\nabla\phi - \nabla I_h\phi)\Big)\;ds\\
&-\int_{0}^t \Big( (\bs_\bq,\nabla I_h\phi)-( f_u(u)\phi,P_{k-1}\theta_u)\Big)\;ds,
\end{align*}
where $I_h$ indicates the interpolation operator from $L^2(\Omega)$ into $W_h \cap H_0^1(\mathscr{T}_h)$. Using the second error equation \eqref{ee2} for the second last term in the above equation, we arrive at
\begin{align*}
(\Theta,P_{k-1}\theta_u)
&=(\phi_s(0),P_{k-1}\theta_u(0))+\int_0^t\Big( -(P_{k-1}\phi,\theta_{u_{ss}})+(\bq_h-\bq,{\bf{\Pi}}^{BDM}\nabla\phi - \nabla I_h\phi)+(\theta_{u_{ss}},I_h\phi)\\
&+(f(u)-f(u_h),I_h\phi)-(\pw u_{ss}-u_{ss},I_h\phi)-( f_u(u)\phi,P_{k-1}\theta_u)\Big)\;ds.
\end{align*}
Here, as  $\langle\widehat{\bs}_\bq,\mu\rangle = 0$ on $\mathcal{E}_h/\Gamma_\partial$ and $I_h\phi = 0$ on $\partial\Omega$, we have used $\langle\widehat{\bs}_{\bq}\cdot\bn,I_h\phi\rangle=0,$
\begin{align}
(\Theta,P_{k-1}\theta_u)
&=(\phi_s(0),P_{k-1}\theta_u(0))+ \int_0^t\Big( (\bq_h-\bq,{\bf{\Pi}}^{BDM}\nabla\phi- \nabla I_h\phi)+(u_{ss}-u_{hss},I_h\phi-P_{k-1}\phi)\Big)\nonumber\\
&+\int_0^t\Big( (f(u)-f(u_h),I_h\phi)-( f_u(u)\phi,P_{k-1}\theta_u)\Big)ds.\label{post-2}
\end{align}
For the first term in integration, use the identity,
\begin{align}
    \int_0^t h(z)g(z)dz = h(0)\undertilde{g}(0)\;dz + \int_0^t h_z(z)\undertilde{g}(z)\;dz,\label{identity}
\end{align}
and utilize Taylor's series expansion to deal with the term $f(u)-f(u_h)$ on the right-hand side  of \eqref{post-2} 
 to arrive with $ \bE_{\bq s}= \bq_s-\bq_{hs}$ at,
\begin{align}
(\Theta,P_{k-1}\bs_u)
&=(\phi_s(0),P_{k-1}\theta_u(0))- (\bE_{\bq}(0),{\bf{\Pi}}^{BDM}\nabla\undertilde{\phi}(0)- \nabla I_h\undertilde{\phi}(0))
- \int_0^t\Big((\bE_{\bq s},{\bf{\Pi}}^{BDM}\nabla\undertilde{\phi}- \nabla I_h\undertilde{\phi}) \nonumber\\
&+(u_{ss}-u_{hss},I_h\phi-P_{k-1}\phi)+(f_u(u)(u_h-u),I_h\phi-\phi)+(f_u(u)(u_h-u),\phi)\nonumber\\
&+(f_{uu}(u)(u_h-u)^2,I_h\phi)-( f_u(u)\phi,P_{k-1}\theta_u)\Big)ds.\label{post-3}
\end{align}
A use of  the property of $P_{k-1}$ in \eqref{post-3} yields
\begin{align}
(\Theta,P_{k-1}\bs_u)
=& (\phi_s(0),P_{k-1}\theta_u(0))-  (\bE_{\bq}(0),{\bf{\Pi}}^{BDM}\nabla\undertilde{\phi}(0)- \nabla I_h\undertilde{\phi}(0))-\int_0^t\Big( (\bE_{\bq s},{\bf{\Pi}}^{BDM}\nabla\undertilde{\phi}- \nabla I_h\undertilde{\phi})\nonumber\\&+(u_{ss}-u_{hss},I_h\phi-P_{k-1}\phi)+(f_u(u)(u_h-u),I_h\phi-\phi)+(f_{uu}(u)(u_h-u)^2,I_h\phi)\nonumber\\&-( f_u(u)\phi-P_{k-1}(f_u(u)\phi),P_{k-1}\bs_u+\bs_u)+(f_u(u)\phi-P_{k-1}(f_u(u)\phi),\br_u)\Big)ds\nonumber\\
 =&J_0+J_1 +\int_0^t\big( J_2 + J_3 +J_4 + J_5+J_6+J_7\big)\;ds. \label{pp1}
\end{align}
A straightforward applications of the Cauchy-Schwarz inequality  with \eqref{estimate:IC}-\eqref{estimate:IC-0}   elliptic regularity results yield
\begin{align*}
&|J_0| \leqslant \|\phi_s(0)\| \;\|\theta_u(0)\| \leqslant \;h^{k+2}\; \|u_0\|_{H^{k+2}}\;\|\phi_s(0)\| \\
&|J_1| \leqslant \|{\bf{\Pi}}^{BDM}\nabla\undertilde{\phi}(0)- \nabla I_h\undertilde{\phi}(0)\| \|\bq_h(0)-\bq(0)\|\leqslant  C\;h^{k+2} \|\undertilde{\phi}(0)\|_{H^2},
\end{align*} 
A use of  Theorem~\ref{tt2} with approximation property  in BDM projection \eqref{BDM-projection} and in interpolant shows
\begin{align*}
|J_2| \leqslant &\Big(\|{\bf{\Pi}}^{BDM}\nabla\undertilde{\phi}-\nabla \undertilde{\phi}\|+ \| \nabla (\undertilde{\phi}- I_h\undertilde{\phi})\|\Big)\; \|\bq_{ht}-\bq_t\|\\
&\leqslant C\;h^{k+2} \Big( \|u_0\|_{H^{k+2}} + \|u_1\|_{H^{k+2}} + \|u\|_{W^{3,1}(0,T;H^{k+2})} \Big)\;\|\undertilde{\phi}\|_{H^2}.
\end{align*} 
Similarly an application of Theorem ~\ref{tt2} with approximation property  in BDM projection \eqref{BDM-projection}, $L^2$ projection  and interpolant shows
\begin{align*}
&|J_3| \leqslant \| u_{tt}-u_{htt}\| (\|I_h\phi-\phi\|+\|\phi-P_{k-1}\phi\|)
\leqslant C\;h^{k+2}\Big( \|u_0\|_{H^{k+2}} + \|u_1\|_{H^{k+2}} + \|u\|_{W^{3,1}(0,T;H^{k+2})} \Big)\;\|\phi\|_{H^1},\\
&|J_4| \leqslant \| u-u_h\| \|I_h\phi-\phi\|
\leqslant C\;h^{k+2}\Big( \|u_0\|_{H^{k+2}} + \|u\|_{W^{2,1}(0,T;H^{k+2})} \Big)\;\|\phi\|_{H^1},\\
&|J_5|\leqslant \|u-u_h\|^2\|I_h\phi\|_{L^{\infty}}
\leqslant C\;h^{2k+2}\;\Big( \|u_0\|_{H^{k+2}} + \|u\|_{W^{2,1}(0,T;H^{k+2})} \Big) \;\|\phi\|_{H^1},\\
&|J_6+J_7| \leqslant C\;h\|\phi\|_{H^1}(\|\bs_u\|+\|\br_u\|)
\leqslant C\;h^{k+2}\;\Big( \|u_0\|_{H^{k+2}} + \|u\|_{W^{2,1}(0,T;H^{k+2})} \Big) \;\|\phi\|_{H^1}.
\end{align*}
Substitute all above estimate $J_0,\cdots,J_7$ in \eqref{pp1}, a use of  Proposition $4.1$ yields
\begin{align*}
(\Theta,P_{k-1}\bs_u(t)) &\leqslant  C\;h^{k+2}\Big( \textcolor{blue}{\|u_0\|_{H^{k+2}}}+ \|u_1\|_{H^{k+2}} + \|u\|_{W^{3,1}(0,T;H^{k+2})} \Big)\;\Big(\|\undertilde{\phi}(0)\|_{H^2}+\sup_{s \in (0,T)}\|\phi_s(s)\| \\&+ \int_0^t  (\|\phi\|_{H^1}+\|\undertilde{\phi}\|_{H^2})ds\Big),\\
&\leqslant  C\;h^{k+2}\Big(  \textcolor{blue}{\|u_0\|_{H^{k+2}}}+\|u_1\|_{H^{k+2}} + \|u\|_{W^{3,1}(0,T;H^{k+2})} \Big)\;\|\Theta\|.
\end{align*}
Since 
\begin{align*}
\|P_{k-1}\bs_u(t)\| = \sup_{\Theta \in L^2(\Omega)}\frac{(\Theta,P_{k-1}\bs_u(t))}{\|\Theta\|},
\end{align*}
this concludes the rest of the proof.
\end{proof}
\subsection{Proof of Theorem~\ref{th-pp}} For finding the estimate of $\|u-u_{h}^{*}\|$, we rewrite
\begin{align*}
\|u-u_{h}^{*}\|_K \leqslant \|u-P_{k+1}u\|_K+\|P_{k+1}u-u_{h}^{*}\|_K.
\end{align*}
For ${w}\in H^1(K)$ and $k\geqslant1$, we observe that
\begin{align}
    \|{w}\|_K &\leqslant \|P_0{w}\|_K+\|P_0w-w\|_K
\leqslant \|P_0w\|_K+Ch\|\nabla w\|_K.\label{wk}
\end{align}
Now, set $w = P_{k+1}u-u_{h}^{*}$ and  note that
     \begin{equation*}
         \|P_{0}w\|_K^2 = (P_{k+1}u-u_{h}^{*},P_{0}w)_K
          =(u-u_{h}^{*},P_{0}w)_K.
     \end{equation*}
     Now, with the help of  the HDG projection, we rewrite
     \begin{equation}
\|P_{0}w\|_K^2 = (\bs_u,P_{0}w)_K= (P_{k-1}\bs_u,P_{0}w)_K\Rightarrow \|P_0w\|_K\leqslant \|P_{k-1}\theta_u\|_K.\label{p0w}
     \end{equation}
For any $v\in P_{k+1}(K),$ we note that 
\begin{align}
    (\nabla w, \nabla v)_K &= (\nabla(P_{k+1}u-u_{h}^{*}),\nabla v)_K,\nonumber\\
    &= (\nabla(P_{k+1}u-u),\nabla v)_K+(\nabla(u-u_{h}^{*}),\nabla v)_K,\nonumber\\
    &= (\nabla(P_{k+1}u-u),\nabla v)_K+(\bq-\bq_h),\nabla v)_K, \label{gw}
\end{align}
by the first equation in \eqref{main1}. Hence, substituting the estimates \eqref{p0w} and \eqref{gw} in \eqref{wk}, we obtain that 
 \begin{align*}
     \|u-u_{h}^{*}\|_K \leqslant \|u-P_{k+1}u\|_K+\|P_{k-1}\bs_u\|_K+ Ch(\|\nabla(u-P_{k+1}u)\|_K+\|\bq-\bq_h\|_K).
 \end{align*}
 Now squaring both sides and then using the estimate of $\|\bq-\bq_h\|_K$ from Theorem \ref{tt2} and $\|P_{k-1}\bs_u\|_K$ from \eqref{pp3}, we complete the proof of the Theorem~\ref{th-pp}.

\section{Completely Discrete Schemes}
This section focuses on conservative and non-conservative schemes with error analyses.
\subsection{A Fully Discrete Conservative Scheme}

Let $\Delta t$ be the time step, and $t_n = n\Delta t$, where $1\leqslant n \leqslant N$ and $N\Delta t = 1$. We now define some notations for our subsequent use:
\begin{eqnarray*}
&\varphi^n = \varphi(t_n), &\varphi^{n+1/2}= \frac{1}{2}(\varphi^{n+1}+\varphi^{n}),~~~ \partial^{*}_{t}\varphi^{1/2}=\dfrac{(\varphi^{1}-\varphi^{0})}{\Delta t}\\
&\overline{\partial}_t\varphi^{n} = \dfrac{(\varphi^{n}-\varphi^{n-1})}{\Delta t}, &\partial_{t}\varphi^{n} = \dfrac{(\varphi^{n+1}-\varphi^{n})}{\Delta t},~~~~~~~~
\updelta_{t}\varphi^{n} = \dfrac{\varphi^{n+1}-\varphi^{n-1}}{2\Delta t}\\ &\mathscr{A}{\varphi}^{n}= \frac{1}{2}(\varphi^{n+1}+\varphi^{n-1}),
&\partial_t^{2}\varphi = \frac{1}{(\Delta t)^2}(\varphi^{n+1}-2\varphi^{n}+\varphi^{n-1}) ={\partial}_t \overline\partial_t \varphi^{n}
\end{eqnarray*}
The HDG  fully discretized formulation related to HDG method \eqref{hdg1}-\eqref{hdg5} reads as: Seek $(U^n,\bQ^n,\widehat{U}^n)\in W_h\times \bV_h\times M_h$ such that 
\begin{subequations}\label{ereq2}
\begin{align}
&\frac{2}{\Delta t}(\partial^{*}_{t}U^{1/2},w_h)+(\bQ^{1/2},\nabla w_h)-\langle\widehat {\bQ}^{1/2}\cdot\bn,w_h\rangle+\big(\mathscr{F}(U^{1},U^{0}),w_h\big)-\frac{2}{\Delta t}(u_1,w_h)=0,\label{ep17}\\
&(\partial_{t}^2U^n,w_h)+(\mathscr{A}{\bQ}^{n},\nabla w_h)-\langle\mathscr{A}{\widehat {\bQ}}^{n}\cdot\bn,w_h\rangle+\big(\mathscr{F}(U^{n+1},U^{n-1}),w_h\big)=0,\quad\forall~~  n\geqslant 1, \label{ep11}\\
&(\bQ^{n},\bv_h)+(U^{n},\nabla\cdot\bv_h)-\langle\widehat{U}^{n},\bv_h\cdot\bn\rangle=0,\quad\forall~~ n\geqslant 0, \label{ep12}\\
&\langle \widehat{\bQ}^{n}\cdot\bn,\mu_1\rangle_{\mathcal{E}_h/\Gamma_\partial}=0,\quad\forall~~ n\geqslant 0, \label{ep14}\\
&\langle \widehat{U}^{n},\mu_2\rangle_{\Gamma_\partial}=0,\quad\forall~~ n\geqslant 0,\label{ep15}
\end{align}
\end{subequations}
for all $(w_h,\bv_h,\mu_1, \mu_2)\in W_h\times\bV_h\times M_h\times M_h$, where, 
\begin{equation}
    \widehat{\bQ}^{n}\cdot\bn=\bQ^{n}\cdot\bn- \tau (U^n-\widehat{U}^n)\;\text{ on } \mathcal{E}_h,\label{ep13}
\end{equation}
and the non-linear term is defined as: \[\mathscr{F}(U^{n+1},U^{n-1})=\begin{cases} \dfrac{F(U^{n+1})-F(U^{n-1})}{U^{n+1}-U^{n-1}}, &U^{n+1}\neq U^{n-1}\\ F^{\prime}(U^{n+1}),&U^{n+1}= U^{n-1} \end{cases} \] with $F(U)=\frac{1}{4}(1-U^2)^2$.
Given initial values $U^0$, $Q^0$, and $\widehat{U}^0$, we can determine $U^1$, $Q^1$, and $\widehat{U^1}$ using equations \eqref{ep17} and \eqref{ep12}.
\subsubsection{Discrete energy conservative property. }  Below, we establish the following discrete conservation law for the energy of the fully discrete scheme \eqref{ereq2}.
\begin{theorem}\label{th5}
The fully discrete HDG approximation $(U^{n+1},\bQ^{n+1},\widehat{U}^{n+1})\in W_h\times \bV_h\times M_h$ given by \eqref{ep11}-\eqref{ep14} satisfies the following fully discrete energy conservation property for $n\geqslant 1$,
\begin{eqnarray}
\mathbb{E}^{n+1/2}=\mathbb{E}^{3/2},
\end{eqnarray}
where,
\begin{eqnarray*}
    \mathbb{E}^{n+1/2}=\|\partial_tU^{n}\|^2+\frac{1}{2}\Big(\|\bQ^{n}\|^2+\|\bQ^{n+1}\|^2+\|\tau^{1/2}(U^{n}-\widehat{U}^{n})\|^2+\|U^{n+1}-\widehat{U}^{n+1}\|_+^2\\+2\big(F(U^{n}),1\big)+2\big(F(U^{n+1}),1\big)\Big).
\end{eqnarray*}
\end{theorem}
\begin{proof}
Taking the difference between two levels, i.e., between $t=t_{n+1}$ and $t=t_{n-1}$ and then dividing by $2\Delta t$ in \eqref{ep12}, we find that
\begin{eqnarray}
(\updelta_t\bQ^{n},\bv_h)+(\updelta_tU^{n},\nabla\cdot\bv_h)-\langle\updelta_t\widehat{U}^{n},\bv_h\cdot\bn\rangle&=&0 \label{epp12}
\end{eqnarray}
Now, choose $\bv_h=\mathscr{A}{\bQ}^{n}$, and $w_h=\updelta_t U^n$ in equations \eqref{ep11} and \eqref{epp12} respectively, then add resulting equations
\begin{align*}
(\partial_{t}^2U^n,\updelta_t U^n)+(\updelta_t\bQ^{n},\mathscr{A}{\bQ}^{n})&+(\updelta_tU^{n},\nabla\cdot\mathscr{A}{\bQ}^{n})-\langle\updelta_t\widehat{U}^{n},\mathscr{A}{\bQ}^{n}\cdot\bn\rangle+(\mathscr{A}{\bQ}^{n},\nabla\updelta_t U^n)\\&-\langle\mathscr{A}{\widehat {\bQ}}^{n}\cdot\bn,\updelta_t U^n\rangle+\big(\mathscr{F}(U^{n+1},U^{n-1}),\updelta_t U^n\big)=0.
\end{align*}
An application of \eqref{ep14} with \eqref{ep15} and \eqref{ep13} yields
\begin{align*}
  (\partial_{t}^2U^n,\updelta_t U^n)+(\updelta_t\bQ^{n},\mathscr{A}{\bQ}^{n})+\tau\langle \updelta_t{U}^{n}-\updelta_t\widehat{U}^{n}, \mathscr{A}{U}^n-\mathscr{A}{\widehat U}^n\rangle  +\big(\mathscr{F}(U^{n+1},U^{n-1}),\updelta_t U^n\big) = 0,
\end{align*}
and hence,  we arrive  at 
\begin{align*}
\frac{1}{2\Delta t}\Big(\|\partial_tU^{n}\|^2-\|\partial_tU^{n-1}\|^2&+\frac{1}{2}\Big(\|\bQ^{n+1}\|^2-\|\bQ^{n-1}\|^2+\|U^{n+1}-\widehat{U}^{n+1}\|_+^2\\&-\|U^{n-1}-\widehat{U}^{n-1}\|_+^2\Big) +\big(\mathscr{F}(U^{n+1},U^{n-1}\big),\updelta_t U^n)=0.
\end{align*}
From the definition of $\updelta_t$, it follows that
\begin{align*}
    \big(\mathscr{F}(U^{n+1},U^{n-1}\big),\updelta_t U^n)&=\Big(\dfrac{F(U^{n+1}-F(U^{n-1}))}{U^{n+1}-U^{n-1}},\dfrac{U^{n+1}-U^{n-1}}{2\Delta t}\Big)\\& = \dfrac{1}{2\Delta t}\big(F(U^{n+1})-F(U^{n-1}),1\big).
\end{align*}
Altogether, we obtain
\begin{align}\label{eppp13}
\|\partial_tU^{n}\|^2+\dfrac{1}{2}\big(\|\bQ^{n+1}\|^2&+\|U^{n+1}-\widehat{U}^{n+1}\|_+^2\big)+\big(F(U^{n+1}),1\big)=\|\partial_tU^{n-1}\|^2\nonumber\\&+\dfrac{1}{2}\big(\|\bQ^{n-1}\|^2+\|U^{n-1}-\widehat{U}^{n-1}\|_+^2\big)+\big(F(U^{n-1}),1\big).
\end{align}
By adding 
$$\frac{1}{2}\Big(\|\bQ^{n}\|^2+\|U^{n}-\widehat{U}^{n}\|_+^2+2\big(F(U^{n}),1\big)\Big)$$
 on both sides of \eqref{epp13}, we derive the energy conservation property:
  $$E^{n+1/2}=E^{n-1/2}=...=E^{3/2}.$$
Solving equations \eqref{ep17} and \eqref{ep12} at $n=0$ and add 
$$\frac{1}{2}\Big(\|\bQ^{0}\|^2+\|U^{0}-\widehat{U}^{0}\|_+^2+2\big(F(U^{0}),1\big)\Big)$$ 
on both sides to arrive at
\begin{align}\label{eppp14}
     2\|\partial_t U^{0}\|^2&+\frac{1}{2}\Big(\|\bQ^{0}\|^2+\|\bQ^{1}\|^2+\|U^{0}-\widehat{U}^{0}\|_+^2+\|U^{1}-\widehat{U}^{1}\|_+^2\Big)+\big(F(U^{0}),1\big)\nonumber\\&+\big(F(U^{1}),1\big)=\|\bQ^{0}\|^2+\|U^{0}-\widehat{U}^{0}\|_+^2+2\big(F(U^{0}),1\big)+2(u_1,\partial_tU^{0}).
\end{align}
 From \eqref{eppp14}, we obtain the relation 
 \begin{eqnarray*}
     \mathbb{E}^{3/2}=-\|\partial_tU^{0}\|^2+\|\bQ^{0}\|^2+\|U^{0}-\widehat{U}^{0}\|_+^2+2\big(F(U^{0}),1\big)+2(u_1,\partial_tU^{0}).
\end{eqnarray*}
With the  initial energy at $n=0$ given by 
$$
\mathcal{E}^0=2\|\partial_t U^{0}\|^2+\frac{1}{2}\Big(\|\bQ^{0}\|^2+\|\bQ^{1}\|^2+\|U^{0}-\widehat{U}^{0}\|_+^2+\|U^{1}-\widehat{U}^{1}\|_+^2\Big)+\big(F(U^{0}),1\big)+\big(F(U^{1}),1\big),
$$
 it is observe that $\mathbb{E}^{3/2}\leqslant \mathcal{E}^0$ and this concludes the rest of the proof.
\end{proof}
The next lemma helps us to deal with uniqueness and error estimates.
\begin{lemma}\label{N-Linear}
    For  $\Phi_1, ~\Phi_2,~\Psi_1,~\Psi_2\in L^6(\Omega),$ there holds
    \begin{align*}
        \|\mathscr{F}(\Phi_1,\Psi_1)-\mathscr{F}(\Phi_2,\Psi_2)\|^2\leqslant C\;\big(\|\Phi_1-\Phi_2\|^2_{L^6}+\|\Psi_1-\Psi_2\|^2_{L^6}\big)\big(\|\Phi_1\|^2_{L^6}+\|\Psi_1\|^2_{L^6}+\|\Phi_2\|^2_{L^6}+\|\Psi_2\|^2_{L^6}+1\big).
    \end{align*}
\end{lemma}
\begin{proof}
    From the definition of $\mathscr{F}(U,V)$ with $F(U)=\frac{1}{4}(1-U^2)^2$, it follows that 
    \[\mathscr{F}(\Phi_1,\Psi_1)=\frac{1}{4}(\Phi_1+\Psi_1)(\Phi_1^2+\Psi_1^2-2),\]
 and   therefore,
    \begin{align*}
        \mathscr{F}(\Phi_1,\Psi_1)-\mathscr{F}(\Phi_2,\Psi_2)&=\dfrac{1}{4}\big(\Phi_1^3-\Phi_2^3+\Psi_1^3-\Psi_2^3+\Phi_1 \Psi_1^2-\Phi_2 \Psi_2+\Phi_1^2 \Psi_1-\Phi_2^2 \Psi_2\\
        &-2(\Phi_1-\Phi_2)-2(\Psi_1-\Psi_2)\big)
    \end{align*}
    Adding and subtracting $\Phi_2 \Psi_1^2$ and $\Phi_1^2 \Psi_2$, we obtain
  \begin{align}
        \mathscr{F}(\Phi_1,\Psi_1)-\mathscr{F}(\Phi_2, \Psi_2)&=\dfrac{1}{4}\Big(\big(\Phi_1^3-\Phi_2^3\big)+\Psi_1^2\big(\Phi_1-\Phi_2\big)+\Psi_2\big(\Phi_1^2-\Phi_2^2\big)-2(\Phi_1-\Phi_2)\nonumber\\
        &+\big(\Psi_1^3-\Psi_2^3\big)+\Phi_1^2\big(\Psi_1-\Psi_2\big)
        +\Phi_2\big(\Psi_1^2-\Psi_2^2\big)-2(\Psi_1-\Psi_2)\Big)
    \end{align}
 A use of the Cauchy-Schwarz inequality with the property of $L^p$ bounds yields    
 \begin{align}
        \|\mathscr{F}(\Phi_1,\Psi_1)-\mathscr{F}(\Phi_2,\Psi_2)\|^2
        &\leqslant \frac{1}{16}\|\Big(\big(\Phi_1-\Phi_2\big)+\big(\Psi_1-\Psi_2\big)\Big)\big(\Phi_1+\Psi_1+\Phi_2+\Psi_2+2\big)^2\|^2 \nonumber\\
        &\leqslant C\;\big(\|\Phi_1-\Phi_2\|^2_{L^6}+\|\Psi_1-\Psi_2\|^2_{L^6}\big)\big(\|\Phi_1\|^2_{L^6}\nonumber\\
        &+\|\Psi_1\|^2_{L^6}+\|\Phi_2\|^2_{L^6}+\|\Psi_2\|^2_{L^6}+1\big),
    \end{align}
and this completes the rest of the proof.
    \end{proof}
\subsection{Well-posedness of Discrete Problem}
The discrete system \ref{ereq2} gives rise to a system of nonlinear algebraic equations, requiring a discussion of its well-posedness using the following variant of the Brouwer fixed point theorem.
\begin{lemma}[The Brouwer fixed point theorem \cite{kesavan2015topics}]\label{BFT}
    Let $X$ be a finite-dimensional Hilbert space with inner product $\langle\cdot,\cdot\rangle$ and norm $\|\cdot\|$. Further let $\mathscr{H}: X\rightarrow X$ be a continuous map such that $\big(\mathscr{H}(x),x\big)>0~ \forall x\in X$ with $\|x\|_X=R>0.$ Then, there exist $x^* \in X$ with $\|x^*\|_X<R$ such that $\mathscr{H}(x^*)=0.$
\end{lemma}
\begin{theorem} [Existence and Uniqueness Result]\label{EU-DS}
Given $(U^m,\bQ^m,\widehat{U}^m)\in W_h\times \bV_h\times M_h$ for $m=0,1,\cdots,n,$  there exists a unique discrete solution triplet $(U^{n+1},\bQ^{n+1},\widehat{U}^{n+1})\in W_h\times \bV_h\times M_h$ to \eqref{ereq2}.
\end{theorem}
\begin{proof}
In order to apply Lemma ~\ref{BFT}, now
re-write system \eqref{ereq2} as
\begin{subequations}\label{exi01}
\begin{align}
     (\dfrac{2\mathscr{A}{U}^{n}-2U^n}{\Delta t^2},w_h)+(\mathscr{A}{\bQ}^{n},\nabla w_h)-\langle\mathscr{A}{\widehat {\bQ}}^{n}\cdot\bn,w_h\rangle+\big(\mathscr{F}(2\mathscr{A}{U}^{n}-U^{n-1},U^{n-1}),w_h\big)=0,\label{exi1}\\
    (\mathscr{A}\bQ^{n},\bv_h)+(\mathscr{A}U^{n},\nabla\cdot\bv_h)-\langle\mathscr{A}\widehat{U}^{n},\bv_h\cdot\bn\rangle=0,\label{exi2}\\
    \langle \mathscr{A}\widehat{\bQ}^{n}\cdot\bn,\mu_1\rangle_{\mathcal{E}_h/\Gamma_\partial}=0, \label{exi3}\\
\langle \mathscr{A}\widehat{U}^{n},\mu_2\rangle_{\Gamma_\partial}=0,\label{exi4}
\end{align}
\end{subequations}
and
\begin{equation}
    \mathscr{A}\widehat{\bQ}^{n}\cdot\bn=\mathscr{A}\bQ^{n}\cdot\bn- \tau (\mathscr{A}U^n-\mathscr{A}\widehat{U}^n)\;\text{ on } \mathcal{E}_h,\label{exi5}
\end{equation}
With $W_1 = \mathscr{A}U^{n},~\bW = \mathscr{A}{\bQ}^{n}$ and $\widehat{W}_1 = \mathscr{A}\widehat{U}^{n},$
we need to discuss first the existence of solution triplet $\mathscr{W} =(W_1, \bW, \widehat{W}_1)\in W_h\times \bV_h\times M_h$ using Lemma~\ref{BFT}. Now a use of the flux equation \eqref{exi5} with \eqref{exi3} and \eqref{exi4} shows with $x=\mathscr{W}$ and $\big(\mathscr{H}(x),x\big)$ as
\begin{equation}
 \begin{split}
 \big(\mathscr{H}(\mathscr{W}), \mathscr{W}\big) :&=
    (\dfrac{2W_1-2U^n}{\Delta t^2},W_1)+(\bW,\bW)+(\bW,\nabla W_1)-\langle\bW\cdot\bn,W_1\rangle+\langle \tau(W_1-\widehat{W}_1),W_1\rangle\nonumber\\&+\big(\mathscr{F}(2W_1-U^{n-1},U^{n-1}),W_1\big)+(W_1,\nabla\cdot\bW)-\langle\widehat{W}_1,\bW\cdot\bn\rangle,
 \end{split}
 \end{equation}
and hence , using integration by parts, we obtain in a standard way
\begin{equation}
 \begin{split}
 \big(\mathscr{H}(\mathscr{W}), \mathscr{W}\big) :
  &=\|W_1\|^2-(U^n,W_1) +\dfrac{(\Delta t)^2}{2}\big(\|\bW\|^2+\|(W_1-\widehat{W}_1)\|^2_+\big)\nonumber\\
  &- (U^n, W_1)+\dfrac{(\Delta t)^2}{2}\big(\mathscr{F}(2W_1-U^{n-1},U^{n-1}),W_1\big)-\dfrac{\Delta t^2}{2}\langle\widehat{W}_1,\bW\cdot\bn\rangle.
 \end{split}
 \end{equation}
Using the definition of $\mathscr{F}(U,V)$, the  non linear term can be rewritten as 
\[\big(\mathscr{F}(2W_1-U^{n-1},U^{n-1}),W_1\big)=\big(\dfrac{1}{4}(2W_1((2W_1-U^{n-1})^2+{U^{n-1}}^2-2)),W_1\big).\]
and hence, with 
\[\opnorm{\mathscr{W}}^2= \|W_1\|^2+\dfrac{(\Delta t)^2}{2}\Big(\|\bW\|^2+\|(W_1-\widehat{W}_1)\|^2_+\Big)\]
we now derive for $0<\Delta t<1$
\begin{eqnarray}
(\mathscr{H}(\mathscr{W}),\mathscr{W})&=&\opnorm{\mathscr{W}}^2-(U^n,W_1)-\dfrac{(\Delta t)^2}{2}\|W_1\|^2+\dfrac{(\Delta t)^2}{4}\big(W_1((2W_1-U^{n-1})^2+{U^{n-1}}^2),W_1\big),\nonumber\\
&\geqslant& \opnorm{\mathscr{W}}^2-\dfrac{(\Delta t)^2}{2}\|W_1\|^2-\|U^n\|\|W_1\|\nonumber\\
&\geqslant& \big((1-\dfrac{(\Delta t)^2}{2})\opnorm{\mathscr{W}}-\|U^n\|\big)\opnorm{\mathscr{W}}\nonumber\\
&\geqslant& \big(\dfrac{1}{2}\opnorm{\mathscr{W}}-\|U^n\|\big)\opnorm{\mathscr{W}}\nonumber
\end{eqnarray}
Therefore, the map $(\mathscr{H}(\mathscr{W}),\mathscr{W}) > 0$ for all $\opnorm{\mathscr{W}}=2\|U^n\|+1 = R>0$ and now  from Lemma~ \ref{BFT}, there exists $\mathscr{W^*}=(U^{n+1}, \bQ^{n+1}, \widehat{U}^{n+1})\in W_h\times \bV_h\times M_h$ with $\opnorm{(U^{n+1}, \bQ^{n+1}, \widehat{U}^{n+1})}<R$ such that $\mathscr{H}(\mathscr{W^*})=0$. This completes the existence of  solution to the discrete system \eqref{ereq2}.

\noindent For uniqueness of the discrete solutions,  let $(U_1^{n+1}, \bQ_1^{n+1}, \widehat{U}_1^{n+1}),~(U_2^{n+1}, \bQ_2^{n+1}, \widehat{U}_2^{n+1})\in W_h\times \bV_h\times M_h$ be two arbitrary solutions of \eqref{ereq2}. 
With
\begin{equation}
    W^{n+1}_{12}=U_1^{n+1}-U_2^{n+1},\;\;
    \bW^{n+1}=\bQ_1^{n+1}-\bQ_2^{n+1}\;\mbox{and}\;\;
    \widehat{W}^{n+1}_{12}=\widehat{U}_1^{n+1}-\widehat{U}_2^{n+1},
\end{equation}
taking difference appropriately, it follows that
\begin{subequations}\label{ereq2new}
\begin{align}
&(\partial_{t}^2 W^n_{12},w_h)+(\mathscr{A}{\bW}^{n},\nabla w_h)-\langle\mathscr{A}{\widehat {\bW}}^{n}\cdot\bn,w_h\rangle+\big(\mathscr{F}(U_1^{n+1},U_1^{n-1})-\mathscr{F}(U_2^{n+1},U_2^{n-1}),w_h\big)=0, \label{ep11n}\\
&(\bW^{n+1},\bv_h)+(W_{12}^{n+1},\nabla\cdot\bv_h)-\langle\widehat{W}^{n+1}_{12},\bv_h\cdot\bn\rangle=0, 
\label{ep12n}\\
&\langle \widehat{\bW}^{n+1}\cdot\bn,\mu_1\rangle_{\mathcal{E}_h/\Gamma_\partial}=0, \label{ep13n}\\
&\langle \widehat{W}_{12}^{n+1}\cdot\bn,\mu_2\rangle_{\Gamma_\partial}=0,\label{ep14n}
\end{align}
\end{subequations}
for all $(w_h,\bv_h,\mu_1, \mu_2)\in W_h\times\bV_h\times M_h\times M_h$, where, 
\begin{equation}
    \widehat{\bW}^{n+1}\cdot\bn=\bW^{n+1}\cdot\bn- \tau (W^{n+1}_{12}-\widehat{W}^{n+1}_{12}).\label{ep15n}
\end{equation}
Now, choose $w_h=\updelta_t W_{12}^n$ in \eqref{ep11n}.  In \eqref{ep12n}, taking difference between between $t=t_{n+1}$ and $t=t_{n-1}$ and dividing by $2\Delta t,$ set $\bv_h=\mathscr{A}\bW^n.$ Then, add the resulting equations to arrive at 
\begin{equation}\label{ep16n}
 \begin{split}
 \frac{1}{2} \overline{\partial}_t \|\partial_t W^{n}_{12}\|^2
 &+\updelta_t \Big(\|\bW^{n}\|^2
 + \|W^{n}_{12}-\widehat{W}^{n+1}_{12}\|_+^2\Big)\\
 &
 +\big(\mathscr{F}(U_1^{n+1},U_1^{n-1})-\mathscr{F}(U_2^{n+1},U_2^{n-1}),\updelta_t W^n_{12}\big)=0.
 \end{split}
 \end{equation}
 Multiplying by $2\Delta t,$  a use of the Cauchy-Schwarz with Young's inequality shows
  \begin{equation}\label{ep17n}
 \begin{split}
 \|\partial_t W^{n}_{12}\|^2-\|\partial_t W^{n-1}_{12}\|^2&+\frac{1}{2}\Big(\|\bW^{n+1}\|^2-\|\bW^{n-1}\|^2+\|W^{n+1}_{12}-\widehat{W}^{n+1}_{12}\|_+^2
 -\|W^{n-1}_{12}-\widehat{W}^{n-1}_{12}\|_+^2\Big) \\&\leqslant \Delta t\|\big(\mathscr{F}(U_1^{n+1},U_1^{n-1})-\mathscr{F}(U_2^{n+1},U_2^{n-1})\|^2+  \frac{\Delta t}{4}(\|\partial_t W^{n}_{12}\|^2+\|\partial_t W^{n-1}_{12}\|^2).
 \end{split}
 \end{equation}
For the nonlinear term on the right hand side of \eqref{ep16n},  we note   as a consequence of Theorem ~\ref{th5} the boundedness of solution in $L^6$ norm that
\begin{equation}\label{N-1}
 \begin{split}
\|\mathscr{F}(U_1^{n+1},U_1^{n-1})&-\mathscr{F}(U_2^{n+1},U_2^{n-1})\|^2\leqslant C\;\big((\|U_1^{n+1}\|^4_{L^6}
+\|U_1^{n-1}\|^4_{L^6}+\|U_2^{n+1}\|^4_{L^6}\\
&+\|U_2^{n-1}\|^4_{L^6}+2^4)(\|W^{n+1}_{12}\|^2_{L^6}+\|W^{n-1}_{12}\|^2_{L^6})\big),\\
&\leqslant C \big(\|{\bW}^{n+1}\|^2
+\|{\bW}^{n-1}\|^2 +\|W^{n+1}_{12}-\widehat{W}^{n+1}_{12}\|_+^2+\|W^{n-1}_{12}-\widehat{W}^{n-1}_{12}\|_+^2\big).
\end{split}
 \end{equation}
Substituting  \eqref{N-1} in \eqref{ep17n}, and summing up the resulting inequality  from $n=1$ to $m$, we arrive at
\begin{equation}
\begin{split}
 \|\partial_t W_{12}^{m}\|^2&+\frac{1}{2}\Big(\|\bW^{m+1}\|^2+\|W_{12}^{m+1}-\widehat{W}_{12}^{m+1}\|_+^2\Big) 
 \leqslant\|\partial_t W_{12}^{0}\|^2+\frac{1}{2}\big(\|\bW^{0}\|^2+\|W_{12}^{0}-\widehat{W}_{12}^{0}\|_+^2\\&+\|\bW^{1}\|^2+\|W_{12}^{1}-\widehat{W}_{12}^{1}\|_+^2\big)+\Delta t\sum_{n=0}^m\Big( C (\|{\bW}^{n+1}\|^2
+\|{\bW}^{n-1}\|^2 +\|W_{12}^{n+1}-\widehat{W}_{12}^{n+1}\|_+^2)\\&+\|W_{12}^{n-1}-\widehat{W}_{12}^{n-1}\|_+^2+\frac{1}{4}\big(\|\partial_t W_{12}^{n}\|^2+\|\partial_t W_{12}^{n-1}\|^2\big)\Big).
\end{split}
\end{equation}
 With $W^0_{12}=0$, $\bW^{0}=0$ and  $W_{12}^{0}-\widehat{W}_{12}^{0}=0,$ a use of the kickback argument with an
 application of the  discrete Gr\"{o}nwall's lemma yields 
\begin{eqnarray}\label{uniqueness-1}
 (1-\frac{\Delta t}{4})\|\partial_t W^{n}_{12}\|^2&+(\frac{1}{2}-C\Delta t)\big(\|\bW^{n+1}\|^2+\|W^{n+1}_{12}-\widehat{W}^{n+1}_{12}\|_+^2\big)  \leqslant C \Big(\|\partial_t W_{12}^{0}\|^2\nonumber\\
 &+\|\bW^{1}\|^2+\|W_{12}^{1}-\widehat{W}_{12}^{1}\|_+^2\Big).
\end{eqnarray}
Choose $\Delta t$ so that $2C\Delta t < 1$ and $0<\Delta t <1$, and to complete the remaining part, we need to obtain an estimate 
for $n=1.$  Proceed similarly to arrive at
\begin{subequations}
\begin{align}
&\frac{2}{(\Delta t)^2}(W_{12}^1,w_h)-(\frac{1}{2}\bW^{1},\nabla w_h)+\langle\frac{1}{2}\widehat{W}_{12}^{1}\cdot\bn,w_h\rangle+\big(\mathscr{F}(U_1^{1},U_1^{0})-\mathscr{F}(U_2^{1},U_2^{0}),w_h\big)=0, \label{W1}\\
&(\bW^{1},\bv_h)+(W_{12}^{1},\nabla\cdot\bv_h)-\langle\widehat{W}_{12}^{1},\bv_h\cdot\bn\rangle=0, \label{W2}\\
&\langle \widehat{\bW}^{1}\cdot\bn,\mu_1\rangle_{\mathcal{E}_h/\Gamma_\partial}=0, \label{W3}\\
&\langle \widehat{W}_{12}^{1},\mu_2\rangle_{\Gamma_\partial}=0,\label{W4}
\end{align}
\end{subequations}
for all $(w_h,\bv_h,\mu_1, \mu_2)\in W_h\times\bV_h\times M_h\times M_h$, where 
\begin{equation}
    \widehat{\bW}^{1}\cdot\bn={\bW}^{1}\cdot\bn- \tau (W_{12}^{1}-\widehat{W}_{12}^{1}).\label{W5}
\end{equation}
Now, in \eqref{W1}, choosing $w_h=W_{12}^1/(\Delta t)$ and in \eqref{W2} set $\bv_h=\bW^{1}/(\Delta t).$ Then, add these two resulting equations and then multiply by $\Delta t$  to find that 
\begin{equation}
 \begin{split}
2 \|W_{12}^1/(\Delta t)\|^2+\frac{1}{2}\Big(\|\bW^{1}\|^2+\|W_{12}^{1}-\widehat{W}_{12}^{1}\|_+^2\Big) =
-\Delta t \;\big(\mathscr{F}(U_1^{1},U_1^{0})-\mathscr{F}(U_2^{1},U_2^{0}),W_{12}^1/(\Delta t)\big).\label{un11}
 \end{split}
 \end{equation}
For the nonlinear term, using Lemma \ref{N-Linear}, then Lemma \ref{ll2} and Theorem \ref{th5} gives
 \begin{equation}
 \begin{split}
2 \|W_{12}^1/(\Delta t)\|^2+\frac{1}{2}\Big(\|\bW^{1}\|^2+\|W_{12}^{1}-\widehat{W}_{12}^{1}\|_+^2\Big) \leqslant C\;\Delta t\; \big(\|\bW^{1}\|^2+\|W_{12}^{1}-\widehat{W}_{12}^{1}\|_\tau^2\big)+\|W_{12}^1/(\Delta t)\|^2.\label{un21}
 \end{split}
 \end{equation}
 Hence, the use of a kickback argument shows
  \begin{equation}
 \begin{split}
\|W_{12}^1/(\Delta t)\|^2+(\frac{1}{2}-\Delta t \mathrm{C})\Big(\|\bW^{1}\|^2+\|W_{12}^{1}-\widehat{W}_{12}^{1}\|_+^2\Big)\leqslant 0.\label{un31}
 \end{split}
 \end{equation}
 By choosing $\Delta t$ small so that   $(\frac{1}{2}- C \Delta t) ) >0$, there holds  $U_1^1=U_2^1,~ \bQ_1^1=\bQ_2^1,~\widehat{U}_1^{1}=\widehat{U}_2^{1}$.
On substitution in \eqref{uniqueness-1},  this completes the rest of the proof.
\end{proof}
\subsection{ Error analysis}
Since the error estimates for the semidiscrete scheme are derived in Section 3, we  now  split the error in the completely discrete scheme  \eqref{ereq2} as
\begin{align*}
u(t_n)-U^n=u(t_n)-u_h(t_n) +\xi_{u}^{n}\;\;\mbox{and}\;\;
\bq(t_n)-\bQ^n=\bq(t_n)-\bq_h(t_n) +\xi_{\bq}^{n}.
\end{align*}
Now, it remains to estimate $\|\xi_{u}^{n}\|$ and $\|\xi_{\bq}^{n}\|.$ Since  $U^0 = u_{h}(0)$ and $\bQ^0 = \bq_h(0)$ are the solutions of problem \ref{HDG0}, it follows that $\xi_{u}(0)=0, ~\xi_{\bq}(0)=0.$

With the help of \eqref{ereq} and \eqref{ereq2}, we  now obtain
\begin{subequations}
\begin{align}
\dfrac{2}{\Delta t}(\partial^{*}_{t}\xi_{u}^{1/2},w_h)+({\xi}_{\bq}^{1/2},\nabla w_h)-\langle{\widehat{\xi}}_{\bq}^{1/2}\cdot\bn,w_h\rangle+\big(f(u_{h}(t_{1/2}))-\mathscr{F}(U^{1},U^{0}),w_h\big)\nonumber\\+\dfrac{2}{\Delta}(u_1,w_h)=(\dfrac{2}{\Delta}\partial_{t}^* u_{h}(t_{1/2})-{u}_{htt}(t_{1/2}),w_h),\label{ep20}\\
(\xi_{\bq}^{1/2},\bv_h)+(\xi_{u}^{1/2},\nabla\cdot\bv_h)-\langle\widehat{\xi}_{u}^{1/2},\bv_h\cdot\bn\rangle=0,\label{ep10}\\
(\partial_{t}^2\xi_{u}^{n},w_h)+(\mathscr{A}{\xi}_{\bq}^{n},\nabla w_h)-\langle\mathscr{A}{\widehat{\xi}}_{\bq}^{n}\cdot\bn,w_h\rangle+\big(\mathscr{A}f(u_{h}(t_{n}))-\mathscr{F}(U^{n+1},U^{n-1}),w_h\big)\nonumber\\=(\partial_{t}^2 u_{h}(t_{n})-\mathscr{A}{u}_{htt}(t_{n}),w_h),\label{ep21}\\
(\xi_{\bq}^{n},\bv_h)+(\xi_{u}^{n},\nabla\cdot\bv_h)-\langle\widehat{\xi}_{u}^{n},\bv_h\cdot\bn\rangle=0,\label{ep22}
\end{align}
\end{subequations}
Below, we discuss the main theorem on error analysis of the fully discrete scheme. 
\begin{theorem}\label{th5.2}
Let  $\{U^{n},Q^{n},\widehat{U}^{n}\},~n\geqslant 0$ be defined by \eqref{ep17}-\eqref{ep15}. Then, there is a positive constant $C$ independent of $h$ and $k$ such that  for all $ J\geqslant 1$
\begin{align}
       \max_{0\leqslant n \leqslant J}\big( \|\partial_t(u(t^{n+1})-U^{n+1})\|+\|\bq(t^{n+1})-\bQ^{n+1}\|\big)\leqslant C\;\big(h^{k+1}+(\Delta t)^2\big),\label{s1}\\
         \max_{0\leqslant n \leqslant J} \|u(t^{n+1})-U^{n+1}\|\leqslant C\;(h^{k+1}+(\Delta t)^2).\label{s2}
    \end{align}
\end{theorem}
\begin{proof}
Choose $w_h=\updelta_{t}\xi_{u}^{n}$ in \eqref{ep21} and in \eqref{ep22} taking difference between two levels, i.e., between $t=t_{n+1}$ and $t=t_{n-1}$ and then dividing by $2\Delta t$, then choose $\bv_h = \mathscr{A}{\xi}_{\bq}^{n}$, we add the resulting equations. 
\begin{align}\label{eq01}
  \frac{1}{2} \overline{\partial}_t\|\partial_{t}\xi_{u}^{n}\|^2 &+\updelta_t \big(\|\xi_{\bq}^{n}\|^2
  +\|\xi_{u}^{n}-\widehat{\xi}_{u}^{n}\|_+^2\big)
  = (\partial_{t}^2 u_{h}(t_{n})-\mathscr{A}u_{htt}(t_{n}),\updelta_{t}\xi_{u}^{n})\nonumber\\
  &-\big(\mathscr{A}f(u_{h}(t_{n}))-\mathscr{F}(U^{n+1},U^{n-1}),\updelta_{t}\xi_{u}^{n}\big).  
  \end{align}
side of \eqref{eq01} 
We can easily deal with the first term on the right-hand side in \eqref{eq01} using Taylor's series expansion 
\begin{align*}
  \partial_{t}^2 u_{h}(t_{n})-\mathscr{A}u_{htt}(t_{n})&=\partial_{t}^2 (  u_{h}(t_{n})-u(t_{n}) )+\partial_{t}^2 u(t_{n})-\mathscr{A}u_{htt}(t_{n}),\\
  &=-\big(\partial_{t}^2 ( u(t_n)-u_{h}(t_{n}) \big)+\big(\mathscr{A}u_{tt}(t_{n})-\mathscr{A}u_{htt}(t_{n})\big)+\dfrac{1}{(\Delta t)^2}R^{*}_1 -\dfrac{1}{2}R^{*}_2, 
 \end{align*}
 where,
 \begin{align*}
 R^{*}_1=\dfrac{1}{6}\int_{-\Delta t}^{\Delta t}(\Delta t - |\mu|)^3 u_{tttt}(t_n+\mu)\;d\mu,\;\;
R^{*}_2=\int_{-\Delta t}^{\Delta t}(\Delta t - |\mu|) u_{tttt}(t_n+\mu)\;d\mu.
 \end{align*}
 A use of  Theorem \ref{tt2} with Lemma \ref{ll3} and following the argument given in \cite{pani2001mixed} to estimate $R^{*}_1$ and $R^{*}_2$ shows
  \begin{eqnarray} \label{truncation}
   \|\partial_{t}^2 u_{h}(t_{n})-\mathscr{A}u_{htt}(t_{n})\|^2 &\leqslant & C \;h^{2(k+1)} \|u_{tt}\|^2_{L^{\infty}(H^{k+1})}+ \dfrac{1}{(\Delta t)^4}\|R^{*}_1\|^2 +\dfrac{1}{4}\|R^{*}_2\| \nonumber\\
   &\leqslant & C \;\big( h^{2(k+1)} \|u_{tt}\|^2_{L^{\infty}(H^{k+1})}+  (\Delta t)^3 \int_{t_{n-1}}^{t_{n+1}} \|u_{tttt}(s)\|^2\;ds\Big).
  \end{eqnarray}
Therefore, we need to estimate the nonlinear term and 
now, we rewrite it as
\begin{align}\label{ff}
 \mathscr{A}f(u_{h}(t_{n}))-\mathscr{F}(U^{n+1},U^{n-1})&=\mathscr{A}F^\prime\big(u_{h}(t_{n})\big)-\mathscr{F}(U^{n+1},U^{n-1}),\nonumber\\
& = \mathcal{I}_1+\mathcal{I}_2,
\end{align}
where
\begin{align*}
     \mathcal{I}_1= \mathscr{A}F^\prime\big(u_{h}(t_{n})\big)-\mathscr{F}\big(u_h(t_{n+1}),u_h(t_{n-1})\big),\;\;
\mathcal{I}_2=\mathscr{F}\big(u_h(t_{n+1}),u_h(t_{n-1}))\big)-\mathscr{F}(U^{n+1},U^{n-1}). 
\end{align*}
Again, an application of  the Taylor series expansion shows
\begin{align*}
   \mathcal{I}_1& = \dfrac{1}{2}(R_1+R_2)-\dfrac{1}{u_h(t_{n+1})-u_h(t_{n-1})}(R_1-R_2),
   \end{align*}
where, the remainder $R_1$ and $R_2$ are given by
\begin{align*}
    R_1 &= (u_h(t_{n+1})-u_h(t_{n}))^2\int_0^{1} F^{''}(u_h(t_{n})+\lambda(u_h(t_{n+1})-u_h(t_{n})))(1-\lambda)d\lambda,\\
    R_2 &= (u_h(t_{n-1})-u_h(t_{n}))^2\int_0^{1} F^{''}(u_h(t_{n})+\lambda(u_h^{n-1}-u_h(t_{n})))(1-\lambda)d\lambda.
\end{align*}
After substituting  the value of $F^{''}$ and integrating in $\mathcal{I}_1$, we use Theorem \ref{tt3} to derive
\begin{align*}
 \|\mathcal{I}_1\| &\leqslant C\;\big( h^{k+1}+(\Delta t)^2)\big).
\end{align*}
For the estimate of  $\mathcal{I}_2,$ a use of Lemma \ref{N-Linear} with   Lemma \ref{ll2} and Theorem \ref{th5} yields
\begin{align*}
    \|\mathcal{I}_2\|^2 &\leqslant C\;(\|\xi_{u}^{n-1}\|^2_{L^6}+\|\xi_{u}^{n+1}\|^2_{L^6})\big(\|u_h(t_{n+1})\|^4_{L^6}+\|u_h(t_{n-1})\|^4_{L^6}+\|U^{n+1}\|^4_{L^6}+\|U^{n-1}\|^4_{L^6}+1\big) \nonumber\\
 &\leqslant C\; \big(\|\xi_{\bq}^{n+1}\|^2
+\|\xi_{\bq}^{n-1}\|^2 +\|\xi_{u}^{n+1}-\widehat{\xi}_{u}^{n+1}\|_+^2+\|\xi_{u}^{n-1}-\widehat{\xi}_{u}^{n-1}\|_+^2\big).  
    \end{align*}
Substitute the estimates of $ \mathcal{I}_1$ and $\mathcal{I}_2$ in \eqref{ff} to arrive at
\begin{align}
     \| \mathscr{A}f(u_{h}(t_{n}))-\mathscr{F}(U^{n+1},U^{n-1})\|^2&\leqslant C\;\Big((h^{k+1}+(\Delta t)^2)^2+\|\xi_{\bq}^{n+1}\|^2
+\|\xi_{\bq}^{n-1}\|^2 \nonumber\\&+\|\xi_{u}^{n+1}-\widehat{\xi}_{u}^{n+1}\|_+^2+\|\xi_{u}^{n-1}-\widehat{\xi}_{u}^{n-1}\|_+^2\Big).\label{sd}
\end{align}
Substitute  \eqref{truncation} and \eqref{sd} in \eqref{eq01},  multiply by $2 \Delta t$ and sum up from $n =1$ to $m$ with $U^0 =  u_{h}(0),~ Q^0=\bq_h(0)$ and $\widehat{U}^0 = \widehat{u}_h(0)$, to obtain
\begin{align}
 \|\partial_{t}\xi_{u}^{m}\|^2&+\dfrac{1}{2}\big(\|\xi_{\bq}^{m+1}\|^2 +\|\xi_{u}^{m+1}-\widehat{\xi}_{u}^{m+1}\|_+^2\big) \leqslant \dfrac{1}{2}\big(\|\xi_{\bq}^{1}\|^2 +\|\xi_{u}^{1}-\widehat{\xi}_{u}^{1}\|_+^2 \big) + C\; \big( h^{2(k+1)} + (\Delta t)^4\big)\nonumber\\
 &+\frac{1}{4}(\|\partial_{t}\xi_{u}^{n}\|^2+\|\partial_{t}\xi_{u}^{n-1}\|^2)\big)+C\;\Big(\Delta t\sum_{n=1}^m\big(\|\xi_{\bq}^{n+1}\|^2+\|\xi_{\bq}^{n-1}\|^2 \nonumber\\&+\|\xi_{u}^{n+1}-\widehat{\xi}_{u}^{n+1}\|_+^2
 +\|\xi_{u}^{n-1}-\widehat{\xi}_{u}^{n-1}\|_+^2 \big)\Big). 
\end{align}
Using the kickback argument, we arrive at
\begin{align}
(1-\frac{\Delta t}{4})\|\partial_{t}\xi_{u}^{m}\|^2&+(\frac{1}{2}-C\Delta t)\big(\|\xi_{\bq}^{m+1}\|^2 +\|\xi_{u}^{m+1}-\widehat{\xi}_{u}^{m+1}\|_+^2\big) \leqslant \dfrac{1}{2}\big(\|\xi_{\bq}^{1}\|^2 +\|\xi_{u}^{1}-\widehat{\xi}_{u}^{1}\|_+^2 \big)\nonumber\\
&+ C\; \big( h^{2(k+1)} + (\Delta t)^4\big)+C\;\Delta t\sum_{n=1}^m\big(\|\partial_{t}\xi_{u}^{n-1}\|^2 +
\|\xi_{\bq}^{n}\|^2
+\|\xi_{u}^{n}-\widehat{\xi}_{u}^{n}\|_+^2
\big). \label{eq02}
\end{align}
Now, with the help of \eqref{ep20} and \eqref{ep10}, we plan to  bound  $\|\xi_{\bq}^{1}\|^2 +\|\xi_{u}^{1}-\widehat{\xi}_{u}^{1}\|_+^2.$
Choose $w_h=\partial^{*}_{t}\xi_{u}^{1/2}$ in  \eqref{ep20} and $\bv_h={\xi}_{\bq}^{1/2}$ in \eqref{ep10}, after adding resulting equation, we obtain
\begin{align}
   \dfrac{2}{\Delta t}\|\partial^{*}_{t}\xi_{u}^{1/2}\|^2&+\dfrac{1}{2\Delta t}(\|{\xi}_{\bq}^{1}\|^2-\|{\xi}_{\bq}^{0}\|^2+\|\xi_{u}^1-\widehat{\xi}_{u}^1\|_+^2-\|\xi_{u}^0-\widehat{\xi}_{u}^0\|_+^2)\nonumber \\&=(\dfrac{2}{\Delta t}(\partial_{t}^* u_{h}(t_{1/2})- U^0)-{u}_{htt}(t_{1/2}),\partial^{*}_{t}\xi_{u}^{1/2})-(f(u_{h}(t_{1/2}))-\mathscr{F}(U^{1},U^{0}),\partial^{*}_{t}\xi_{u}^{1/2}\big). \label{eq03}
\end{align}
An application of  the Taylor's series expansion in \eqref{eq03} shows
\begin{align}
 \|\xi_{u}^{1}\|^2+ \|\xi_{\bq}^{1}\|^2+\|\xi_{u}^1-\widehat{\xi}_{u}^1\|_+^2\leqslant C\;(h^{2k+2} +  (\Delta t)^4).\label{eq04} 
\end{align}
 On substitution \eqref{eq04} in \eqref{eq02} and choosing $\Delta t >0$ small sothat $(2C\Delta t-1 )>0,$ an application of 
 Gr\"{o}nwall's lemma yields
 \begin{align}\label{estimate-q}
 \max_{0\leqslant m \leqslant J} \big(&\|\partial_{t}\xi_{u}^{m}\|^2+\|\xi_{\bq}^{m+1}\|^2 +\|\xi_{u}^{m+1}-\widehat{\xi}_{u}^{m+1}\|_+^2\big)\leqslant C\;(h^{2k+2} +  (\Delta t)^4).
\end{align}
 For the estimate \eqref{s2},  note using Lemma \ref{ll2} that 
\begin{eqnarray} \label{Lp-estimate}
\|\xi_{u}^{m+1}\|_{L^{p}} \leqslant C \Big( \|\xi_{\bq}^{m+1}\| + \|\xi_{u}^{m+1}-\widehat{\xi}_{u}^{m+1}\|_+ \Big).
\end{eqnarray}
A use of \eqref{estimate-q} in  \eqref{Lp-estimate} for $p=2$  with  Lemma \ref{tt3} and triangle inequality completes the rest of the proof. 
 \end{proof}
 \subsection{A Fully Discrete Non-Conservative Scheme}
This subsection is focused on the fully discrete, non-conservative HDG scheme. We approximate the nonlinear term in system \eqref{ereq2} by replacing $\mathscr{F}(U^{n+1},U^{n-1})$ simply by $f(U^n)$. The advantage of this scheme is that at each time step, a linear system of algebraic equations has to be solved.

The HDG  fully discretized formulation reads as: for each $n\geqslant 0$, we seek $(U^n,\bQ^n,\widehat{U}^n)\in W_h\times \bV_h\times M_h$ such that 
\begin{subequations}\label{ereqq2}
\begin{align}
&\frac{2}{\Delta t}(\partial^{*}_{t}U^{1/2},w_h)+(\bQ^{1/2},\nabla w_h)-\langle\widehat {\bQ}^{1/2}\cdot\bn,w_h\rangle+\big({f}(U^{0}),w_h\big)-\frac{2}{\Delta t}(u_1,w_h)=0\label{epp17}\\
&(\partial_{t}^2U^n,w_h)+(\mathscr{A}{\bQ}^{n},\nabla w_h)-\langle\mathscr{A}{\widehat {\bQ}}^{n}\cdot\bn,w_h\rangle+\big({f}(U^{n}),w_h\big)=0\quad\forall~~  n\geqslant 1, \label{epp11}\\
&(\bQ^{n},\bv_h)+(U^{n},\nabla\cdot\bv_h)-\langle\widehat{U}^{n},\bv\cdot\bn\rangle=0\quad\forall~~ n\geqslant 0, \label{epp13}\\
&\langle \widehat{\bQ}^{n}\cdot\bn,\mu_1\rangle_{\mathcal{E}_h/\Gamma_\partial}=0\quad\forall~~ n\geqslant 0, \label{epp14}\\
&\langle \widehat{U}^{n},\mu_2\rangle_{\Gamma_\partial}=0\quad\forall~~ n\geqslant 0,\label{epp15}
\end{align}
\end{subequations}
for all $(w_h,\bv_h,\mu_1, \mu_2)\in W_h\times\bV_h\times M_h\times M_h$, where, 
\begin{equation}
    \widehat{\bQ}^{n}\cdot\bn=\bQ^{n}\cdot\bn- \tau (U^n-\widehat{U}^n)\;\text{ on } \mathcal{E}_h,
\end{equation}
\begin{theorem}
Assume that the solution triplet $\{U^{n},Q^{n},\widehat{U}^{n}\},~n\geqslant 0$ is defined by the system \eqref{ereqq2}. Then there is a constant $C$ independent of $h$ and $k$ such that for all $ J\geqslant 1,$
\begin{align}
       \max_{0\leqslant n \leqslant J}\big(\|\partial_t(u(t^{n+1})-U^{n+1})\|+\|\bq(t^{n+1})-\bQ^{n+1}\|\big)\leqslant C\;(h^{k+1}+(\Delta t)^2),\label{s3}\\
      \max_{0\leqslant n \leqslant J}\big(\|u(t^{n+1})-U^{n+1}\|\big)\leqslant C\;(h^{k+1}+(\Delta t)^2).\label{s4}
    \end{align}
\end{theorem}
\begin{proof}
    Similar to what is done in Theorem \ref{th5.2}, we arrive at 
\begin{equation}
    \begin{split}
  \frac{1}{2\Delta t}\Big(\|\partial_{t}\xi_{u}^{n}\|^2-\|\partial_{t}\xi_{u}^{n-1}\|^2+\frac{1}{2}\big(\|\xi_{\bq}^{n+1}\|^2&-\|\xi_{\bq}^{n-1}\|^2 +\|\xi_{u}^{n+1}-\widehat{\xi}_{u}^{n+1}\|_+^2-\|\xi_{u}^{n-1}-\widehat{\xi}_{u}^{n-1}\|_+^2\big)\Big) \\&= (\partial_{t}^2 u_{h}(t_{n})-\mathscr{A}u_{htt}(t_{n}),\updelta_{t}\xi_{u}^{n})-\big(\mathscr{A}f(u_{h}(t_{n}))-f(U^{n}),\updelta_{t}\xi_{u}^{n}\big)\label{eq011}  
  \end{split}
\end{equation}\
Only the nonlinear term on the right-hand side of the equation \eqref{eq011} needs to be handled. Therefore, by applying Taylor's expansion, we arrive at
 \begin{eqnarray*}
 |\big(\mathscr{A}f(u_{h}(t_{n}))-f(U^{n}),\updelta_{t}\xi_{u}^{n}\big)|&\leqslant&\|f(u_{h}(t_{n}))-f(U^{n})+\dfrac{1}{2}(r_1+r_2)\|\|\updelta_{t}\xi_{u}^{n}\|,
\end{eqnarray*}
where, $$r_1 = (u_h(t_{n+1})-u_h(t_{n}))^2\int_0^{1} f^{'}(u_h(t_{n})+\lambda(u_h(t_{n+1})-u_h(t_{n})))(1-\lambda)d\lambda,$$
$$r_2 = (u_h(t_{n-1})-u_h(t_{n}))^2\int_0^{1} f^{'}(u_h(t_{n})+\lambda(u_h(t_{n-1})-u_h(t_{n})))(1-\lambda)d\lambda.$$
Note from Lemma \ref{N-Linear} that 
 \begin{eqnarray*}
 \|f(u_{h}(t_{n}))-f(U^{n})\|^2&\leqslant& C \|\xi_{u}^{n}\|^2_{L^6}(\|u_{h}(t_{n})\|^4_{L^6}+\|U^{n}\|^4_{L^6}+1),
\end{eqnarray*}
and using the $L^6$ estimates of $u_h(t_n), U^n$ at $nth$ level with an application of energy arguments, Lemma \ref{ll2} and the Young's inequality yield,
\begin{eqnarray*}
|\big(\mathscr{A}f(u^{n})-f(U^{n}),\updelta_{t}\xi_{u}^{n}\big)| &\leqslant& C\;\big(h^{2k+2}+(\Delta t)^4+\|\xi_{\bq}^{n}\|^2+\|\xi_{u}^{n}-\widehat{\xi}_{u}^{n}\|_+^2+\|\updelta_{t}\xi_{u}^{n}\|^2\big).
\end{eqnarray*}
Putting this inequality in \eqref{eq011}, then multiplying by $2\Delta t$ and summing from $n=1$ to $J$, it follows that
\begin{align*}
 (1-\frac{\Delta t}{4})\|\partial_{t}\xi_{u}^{J}\|^2&+(\frac{1}{2}-C\Delta t)\big(\|\xi_{\bq}^{J+1}\|^2 +\|\xi_{u}^{J+1}-\widehat{\xi}_{u}^{J+1}\|_+^2\big) \leqslant C\;(\Delta t) \Big(\sum_{n=1}^J\big(|h^{2k+2}+(\Delta t)^2|\\&+\|\partial_{t}^2 u_{h}(t_{n})-u_{htt}(t_{n})\|^2+\|\partial_{t}\xi_{u}^{n}\|^2+\|\xi_{\bq}^{n}\|^2 +\|\xi_{u}^{n}-\widehat{\xi}_{u}^{n}\|_+^2\big)+\|\xi_{\bq}^{1}\|^2 +\|\xi_{u}^{1}-\widehat{\xi}_{u}^{1}\|_+^2\Big).
\end{align*}
Now, for $\Delta t >0$ small, $1-2C\Delta t >0$. The second term on the right-hand side of \eqref{eq011} is one that we previously dealt with in the Theorem \ref{th5.2}. The desired outcome is now obtained by utilizing Gr\"{o}nwall's inequality. The last inequality \eqref{s4} in this theorem can be proved similarly as in the conservative case. This concludes the rest of the proof.
 \end{proof}
 \section{One Variation  of  the HDG method}
This section deals with a novel HDG method proposed by Qiu and Shi in \cite{qiu2018hdg}, a variation of the HDG method. It differs from the existing method in that polynomials of degree $k+1$ and $k$, with $k\geqslant 0$, approximate the displacement and its gradient, respectively. As a result,  the convergence of order $O(h^{k+2})$ for the displacement in $L^{\infty}(L^2)$-norm is derived, while the convergence of order $O(h^{k+1})$ is shown for the gradient in $L^{\infty}(L^2)$-norm. 

In this section, we indicate only the differences in the proof of the error analysis. 
With $\bV_h$ and $M_h$ as in Section 2, the space $W_h$ is now changed to
$$W_h=\{w\in L^2(\Omega):w|_K\in \mathcal{P}_{k+1}(K)~\forall ~K\in \mathscr{T}_h\}$$
The new variation of the  HDG formulation is to seek $(u_h(t),\bq_h(t),\widehat{u}_h(t)) \in W_h\times \bV_h\times M_h$  such that for $t\in (0,T]$
\begin{subequations}
\begin{eqnarray}
(\bq_h,\bv_h)+(u_h,\nabla\cdot\bv_h)-\langle\widehat{u}_h,\bv_h\cdot\bn\rangle&=&0\quad\forall \bv_h\in\bV_h,\label{nhdg1}\\
(u_{htt},w_h)+(\bq_h,\nabla w_h)-\langle\widehat{\bq}_h\cdot\bn,w_h\rangle+(f(u_h),w_h)&=&0\quad\forall w_h\in W_h,\label{nhdg2}\\
\langle \widehat{\bq}_h\cdot\textbf{n},\mu_1\rangle_{\mathcal{E}_h/\Gamma_\partial}&=&0\quad\forall\mu_1\in M_h,\label{nhdg3}\\
\langle \widehat{u}_h,\mu_2\rangle_{\Gamma_\partial}&=&0\quad\forall\mu_2\in M_h,\label{nhdg4}
\end{eqnarray}
\end{subequations}
where the new flux is defined as
\begin{equation}
\widehat{\bq}_h\cdot \bn=\bq_h\cdot\bn+\dfrac{\tau}{h_K}(\mathcal{P}u_h-\widehat{u}_h) \text{ on } \mathcal{E}_h.\label{nhdg5}
\end{equation}
Here, $\mathcal{P}$ is the $L^2$ projection onto $M_h.$\\
Energy conservation is the same as in Theorem \ref{2.1} with Corollary \ref{coro}.
\begin{theorem}\label{newth}
  Let $u_0,~u_1 \in L^{\infty}(0,T;H^{k+2})$ and $u \in  L^{\infty}(0,T;H^{k+2}\cap H^1_0)$, $u_t \in L^{2}(0,T; H^{k+2})$. Then for all $h= \max\limits_{K\in \mathscr{T}_h}h_K$ and for all $t \in (0,T]$, the following estimates hold:
\begin{equation}\label{cgt-NHDG}
 \begin{split}
 \|(u_t-u_{ht})(t)\|+\|(\bq-\bq_h)(t)\|+\|\nabla(u-u_h)(t)\| \leqslant {C\;h^{k+1} \Big(\|u_0\|_{H^{k+2}} +\|u\|_{L^{\infty}(0,T;H^{k+2})}+ \|u_t\|_{L^{2}(0,T;H^{k+2})}}\Big).
 \end{split}
\end{equation}
with the additional assumptions $u_{tt}\in L^2(0,T;H^{k+2}),$
\begin{equation}\label{cgt-NHDG1}
 \begin{split}
 {\|(u_{tt}-u_{htt})(t)\|+\|(\bq_t-\bq_{ht})(t)\| \leqslant C\;h^{k+1} \Big(\|u_1\|_{H^{k+2}} +\|u_t\|_{L^{\infty}(0,T;H^{k+2})} + \|u_{tt}\|_{ L^2(0,T;H^{k+2})} \Big).}
 \end{split}
\end{equation}
Moreover,
\begin{equation}\label{super-cgt-NHDG}
 \begin{split}
\|(u-u_h)(t)\|\leqslant { C\;h^{k+2}\Big(\|u_0\|_{H^{k+2}} +\|u_1\|_{H^{k+2}} + \|u\|_{L^{\infty}(0,T;H^{k+2})} +\|u_{tt}\|_{ L^2(0,T;H^{k+2})} \Big).}
\end{split}
\end{equation}  
\end{theorem}
Proof of this Theorem is given in the next two subsections.
\subsection{Error Estimates} 
We redefine the projections $\pw$ and $\pV$ as the standard $L^2$ orthogonal projection onto $W_h$ and $\bV_h$, respectively, and followings are the properties of these projections \cite{qiu2018hdg}:
\begin{equation}\left.
\begin{aligned}
    \|\br_u\| &\leqslant C h^{l_u} \|u\|_{H^{l_u}(\mathscr{T}_h)}, && \quad \forall \quad 0 \leqslant l_u \leqslant k+2, \\
    \|\widehat{\br}_u\|_{\mathcal{E}_h} &\leqslant C h^{l_u - 1/2} \|\bq\|_{H^{l_u}(\mathscr{T}_h)}, && \quad \forall \quad 1 \leqslant l_u \leqslant k+1, \\
    \|\br_u\|_{E} &\leqslant C h^{l_u - 1/2} \|u\|_{H^{l_u}(K)}, && \quad \forall \quad 1 \leqslant l_u \leqslant k+2, \\
    \|\br_\bq\| &\leqslant C h^{l_q} \|\bq\|_{H^{l_q}(\mathscr{T}_h)}, && \quad \forall \quad 0 \leqslant l_q \leqslant k+1, \\
    \|\widehat{\br}_\bq\|_{E} &\leqslant C h^{l_q - 1/2} \|\bq\|_{H^{l_q}(K)}, && \quad \forall \quad 1 \leqslant l_q \leqslant k+1, \\
    \|\br_\bq\|_{E} &\leqslant C h^{l_q - 1/2} \|\bq\|_{H^{l_q}(K)}, && \quad \forall \quad 1 \leqslant l_q \leqslant k+1, \\
    \|w\|_{E} &\leqslant C h^{1/2} \|w\|_K, && \quad \forall \quad w \in W_h.
\end{aligned}\right\}\label{L2proj}
\end{equation}
Here, {$E$} denotes the face of an element $K$. Also, define  
\[\|\psi\|^2 = \sum_{K\in\mathscr{T}_h} \|\psi\|^2_{L^2(K)}, ~{\|\psi-\widehat{\psi}\|_{\partial\mathscr{T}_h}^2 = \sum_{K\in\mathscr{T}_h} \left\Vert\left(\dfrac{\tau}{h_K}\right)^{1/2}(\mathcal{P}\psi-\widehat{\psi})\right\Vert_{L^2(\partial K)}}.\]
Using $L^2$ projections $(\pw, \pV, \mathcal{P})$, error equations in $\bs_u, ~\bs_\bq$ and $\widehat\bs_{u}$ are written for all $(\bv_h,w_h,\mu_1, \mu_2)\in  \bV_h\times W_h\times M_h\times M_h$ as
\begin{subequations}\label{ereq1}
\begin{eqnarray}
(\bs_\bq,\bv_h)+(\bs_u,\nabla\cdot\bv_h)-\langle \widehat{\bs}_u,\bv_h\cdot\bn\rangle&=&0,\label{nee1}\\
(\bs_{u_{tt}},w_h)+(\bs_\bq,\nabla w_h)-\langle(\bq-\widehat{\bq}_h)\cdot\bn,w_h\rangle+(f(u)-f(u_h),w_h)
&=&0,\label{nee2}\\
\langle (\bq-\widehat{\bq}_h)\cdot\bn,\mu_1\rangle_{\mathcal{E}_h/\Gamma_\partial}&=&0,\label{nee4}\\
\langle\widehat{\bs}_u,\mu_2\rangle_{\Gamma_\partial}&=&0,\label{nee5}
\end{eqnarray}
\end{subequations}
where the numerical trace for flux is defined as
\begin{align}\label{neflux}
    (\bq-\widehat{\bq}_h)\cdot\bn=(\bs_{\bq}-\br_\bq)\cdot\bn-\dfrac{\tau}{h_K}(\mathcal{P}\bs_u-\widehat{\bs}_u)+\dfrac{\tau}{h_K}\mathcal{P}\br_u.
\end{align}
With the help of the error equations \eqref{ereq1}, we derive the estimates of $\bs_\bq$. Differentiating equation \eqref{nee1} with respect to time and set $\bv_h= \bs_{\bq}$ and choose $w_h= \bs_{u_t}$ in \eqref{nee2}, then adding these resulting equations to obtain
\begin{align}
     \frac{1}{2}\frac{d}{dt}\left(\|\bs_{\bq}\|^2+\|\theta_{u_{t}}\|^2\right)+\langle
     {\bs}_\bq\cdot\bn,\bs_{u_t}-\widehat{\bs}_{u_t}\rangle-\langle(\bq-\widehat{\bq}_h)\cdot\bn,\bs_{u_t}\rangle+(f(u)-f(u_h),\theta_{u_{t}})=0.
\end{align}
Adding the zero-term $\langle(\bq-\widehat{\bq}_h)\cdot\bn,\widehat{\bs}_{u_t}\rangle$, a use of   \eqref{neflux} shows
\begin{align*}
     \frac{1}{2}\frac{d}{dt}\left(\|\bs_{\bq}\|^2+\|\theta_{u_{t}}\|^2+{\|\bs_u-\widehat{\bs}_{u}\|_{\partial\mathscr{T}_h}^2}\right)+\left\langle
     {\br}_{\bq}\cdot\bn +\dfrac{\tau}{h_K}\mathcal{P}\br_u,\bs_{u_t}-\widehat{\bs}_{u_t}\right\rangle+(f(u)-f(u_h),\theta_{u_{t}})=0.
\end{align*}
On integrating from $0$ to $t$, it follows that
\begin{align}
     \|\bs_{\bq}\|^2+\|\theta_{u_{t}}\|^2+{\|\bs_u-\widehat{\bs}_{u}\|_{\partial\mathscr{T}_h}^2}=-\int_0^t\Bigg(\left\langle
     {\br}_{\bq}\cdot\bn +\dfrac{\tau}{h_K}\mathcal{P}\br_u,\bs_{u_t}-\widehat{\bs}_{u_t}\right\rangle+(f(u)-f(u_h),\theta_{u_{t}})\Bigg)\;ds.\label{57}
\end{align}
For estimating the right-hand side terms in \eqref{57}, we first rewrite  the first term as
\begin{align*}
    \int_0^t\Bigg(\left\langle{\br}_{\bq}\cdot\bn+\dfrac{\tau}{h_K}\mathcal{P}\br_u,\bs_{u_t}-\widehat{\bs}_{u_t}\right\rangle\Bigg)\;ds =&\int_0^t\Bigg(\dfrac{d}{ds}\left\langle{\br}_{\bq}\cdot\bn+\dfrac{\tau}{h_K}\mathcal{P}\br_u,\bs_{u}-\widehat{\bs}_{u}\right\rangle\\&-\left\langle{\br}_{\bq_t}\cdot\bn+\dfrac{\tau}{h_K}\mathcal{P}\br_{u_t},\bs_{u}-\widehat{\bs}_{u}\right\rangle\Bigg)\;ds,\\
   =&\left\langle{\br}_{\bq}\cdot\bn+\dfrac{\tau}{h_K}\mathcal{P}\br_u,\bs_{u}-\widehat{\bs}_{u}\right\rangle(t)-\left\langle{\br}_{\bq}\cdot\bn+\dfrac{\tau}{h_K}\mathcal{P}\br_u,\bs_{u}-\widehat{\bs}_{u}\right\rangle(0)\\&-\int_0^t\left\langle{\br}_{\bq_t}\cdot\bn+\dfrac{\tau}{h_K}\mathcal{P}\br_{u_t},\bs_{u}-\widehat{\bs}_{u}\right\rangle\;ds.
    \end{align*}
Using the properties of $L^2$ projection $\mathcal{P}$, the Cauchy-Schwarz inequality, and the trace's inequality, we arrive at
\begin{align}
  \langle{\br}_{\bq}\cdot\bn,\bs_{u}-\widehat{\bs}_{u}\rangle&\leqslant h^{1/2}\|\br_\bq\cdot\bn\|_{\partial\mathscr{T}_h}\|h_{K}^{-1/2}(\bs_{u}-\widehat{\bs}_{u})\|_{\partial\mathscr{T}_h}\nonumber\\&\leqslant {C\;h^{l_q}\|\bq\|_{H^{l_q}(\mathscr{T}_h)}\|\bs_{u}-\widehat{\bs}_{u}\|_{\partial\mathscr{T}_h},~~~ 0\leqslant l_q \leqslant k+1} \label{58}
\end{align}
Similarly, 
\begin{align}
    \left\langle\dfrac{\tau}{h_K}\mathcal{P}\br_u,\bs_{u}-\widehat{\bs}_{u}\right\rangle
    &\leqslant \|\tau h_{K}^{-1/2}\br_u\|_{\partial\mathscr{T}_h}\|h_{K}^{-1/2}(\bs_{u}-\widehat{\bs}_{u})\|_{\partial\mathscr{T}_h}\nonumber\\
    &\leqslant {C\; h^{l_u-1}\|u\|_{H^{l_u}(\mathscr{T}_h)}\|\bs_{u}-\widehat{\bs}_{u}\|_{\partial\mathscr{T}_h},~~~1\leqslant l_u \leqslant k+2}\label{59}
\end{align}
The non-linear term is dealt with similarly as done in Section 3. Hence, we find that
\begin{align}
 \|f(u)-f(u_h)\|^2 &\leqslant C\;\big(\|\br_u\|^2+\|\theta_{\bq}\|^2 + {\|\theta_u-\widehat{\theta}_u\|_{\partial\mathscr{T}_h}^2}\big). \label{150}
 \end{align}
 Substituting \eqref{58}, \eqref{59} and \eqref{150} in \eqref{57}, and using the Young's inequality with kickback argument, we  arrive at
\begin{align}
  \|\bs_{\bq}\|^2+\|\theta_{u_{t}}\|^2+\|\bs_u-\widehat{\bs}_{u}\|_{\partial\mathscr{T}_h}^2\leqslant & C\;\Big(h^{2l_q}(\|\bq\|^2_{H^{l_q}(\mathscr{T}_h)}+\|\nabla u_0\|^2_{H^{l_q}(\mathscr{T}_h)})+h^{2l_u-2}(\|u\|^2_{H^{l_u}(\mathscr{T}_h)}+\|u_0\|^2_{H^{l_u}(\mathscr{T}_h)})\Big)
  \nonumber\\
  &+ C\;\int_0^t\big(h^{2l_u-2}\|u_t\|^2_{H^{l_u}(\mathscr{T}_h)} +h^{2l_q}\|\bq_t\|^2_{H^{l_q}(\mathscr{T}_h)} 
  +h^{2l_u}\|u\|^2_{H^{l_u}(\mathscr{T}_h)}  \big)\;ds  \nonumber\\
  & +C\;\int_0^t\big(\|\theta_{\bq}\|^2+\|\theta_{u_{t}}\|^2+ \|\theta_u-\widehat{\theta}_u\|_{\partial\mathscr{T}_h}^2\big)\;ds.
   \end{align}
Then, an application of Gr\"{o}nwall's inequality  $ 1\leqslant l_u \leqslant k+2 $ and $ 0\leqslant l_q \leqslant k+1,$  yields
\begin{align}
     \|\bs_{\bq}\|^2+\|\theta_{u_{t}}\|^2+ \|\bs_u-\widehat{\bs}_{u}\|_{\partial\mathscr{T}_h}^2\leqslant C\;h^{2(k+1)} \Big( \|u_0\|^2_{H^{k+2}}+
     \|u\|^2_{H^{k+2}}
     + \int_0^t \big(\|u_t\|^2_{H^{k+2}} +\|\bq_t\|^2_{H^{k+1}}+ \|u\|^2_{H^{k+2}}\big)\;ds\Big).
     \label{new}
\end{align}
Next, a  use of the estimate of $\bs_u$ from article {\cite[Lemma 3.2]{qiu2016hdg}} shows
\begin{align*}
    \|\bs_u\|^2  &\leqslant C \;\big(\|\nabla \bs_u\|^2+\|h_{K}^{-1/2}(\bs_{u}-\widehat{\bs}_{u})\|^2_{\partial\mathscr{T}_h}\big),\\
     &\leqslant C\; \big(\|\bs_\bq\|^2+\|h_{K}^{-1/2}(\mathcal{P}\theta_u-\widehat{\theta}_u)\|^2_{\partial\mathscr{T}_h}\big),\\
   & \leqslant C\; \big(\|\bs_\bq\|^2+\|\theta_{u_{t}}\|^2+{\|\bs_u-\widehat{\bs}_{u}\|_{\partial\mathscr{T}_h}^2}\big)
\end{align*}
Hence, we derive the optimal estimates of $\|\bs_u\|$  from \eqref{new}. A use of Lemma~\ref{IC-approx-1} and Theorem~\ref{tt1} in \eqref{new} completes the proof of \eqref{cgt-NHDG}.

{In order to prove \eqref{cgt-NHDG1}, differentiate \eqref{nee1} twice with respect to time and differentiate \eqref{nee2}-\eqref{nee5} and then choose $\bv_h:=\bs_{\bq_t}$ and $w_h:=\theta_{u_{tt}}$ to obtain
\begin{align}
     \frac{1}{2}\frac{d}{dt}\left(\|\bs_{\bq_t}\|^2+\|\theta_{u_{tt}}\|^2\right)+\langle
     {\bs}_{\bq_t}\cdot\bn,\bs_{u_{tt}}-\widehat{\bs}_{u_{tt}}\rangle-\langle(\bq_t-\widehat{\bq}_{ht})\cdot\bn,\bs_{u_{tt}}\rangle+(f(u_t)-f(u_{ht}),\theta_{u_{tt}})=0.
\end{align}
Use of \eqref{neflux}, after differentiating with respect to time shows
\begin{align*}
     \frac{1}{2}\frac{d}{dt}\left(\|\bs_{\bq_t}\|^2+\|\theta_{u_{tt}}\|^2+{\|\bs_{u_t}-\widehat{\bs}_{u_t}\|_{\partial\mathscr{T}_h}^2}\right)+\left\langle
     {\br}_{\bq_t}\cdot\bn+\dfrac{\tau}{h_K}\mathcal{P}\br_{u_t},\bs_{u_{tt}}-\widehat{\bs}_{u_{tt}}\right\rangle+(f(u_t)-f(u_{ht}),\theta_{u_{tt}})=0.
\end{align*}
Proceed similarly as done in proving \eqref{cgt-NHDG},
\begin{align*}
    \int_0^t\Bigg(\left\langle
     {\br}_{\bq_t}\cdot\bn+\dfrac{\tau}{h_K}\mathcal{P}\br_{u_t},\bs_{u_{tt}}-\widehat{\bs}_{u_{tt}}\right\rangle\Bigg)\;ds = &
   \left\langle{\br}_{\bq_t}\cdot\bn+\dfrac{\tau}{h_K}\mathcal{P}\br_{u_t},\bs_{u_t}-\widehat{\bs}_{u_t}\right\rangle(t)\\&-\left\langle{\br}_{\bq_t}\cdot\bn+\dfrac{\tau}{h_K}\mathcal{P}\br_{u_t},\bs_{u_t}-\widehat{\bs}_{u_t}\right\rangle(0)\\&-\int_0^t\left\langle{\br}_{\bq_{tt}}\cdot\bn+\dfrac{\tau}{h_K}\mathcal{P}\br_{u_{tt}},\bs_{u_t}-\widehat{\bs}_{u_t}\right\rangle\;ds.
    \end{align*}
    Hence
 \begin{align*} 
   \left|\left\langle{\br}_{\bq_t}\cdot\bn+\dfrac{\tau}{h_K}\br_{u_t},\bs_{u_t}-\widehat{\bs}_{u_t}\right\rangle\right|\leqslant C\;(h^{l_q}\|\bq_t\|_{H^{l_q}(\mathscr{T}_h)}+ h^{l_u-1}\|u_t\|_{H^{l_u}(\mathscr{T}_h)})\|\bs_{u_t}-\widehat{\bs}_{u_t}\|_{\partial\mathscr{T}_h},\\ \forall ~0\leqslant l_q \leqslant k+1 ,~ 1\leqslant l_u \leqslant k+2.
    \end{align*}  
The non-linear term is dealt with similarly to Section 3. Thus, using the Young's inequality with kickback argument,
\begin{align*}
     \|\bs_{\bq_t}\|^2+\|\theta_{u_{tt}}\|^2+\|\bs_{u_t}-\widehat{\bs}_{u_t}\|_{\partial\mathscr{T}_h}^2\leqslant& C\;\Big(h^{2l_q}(\|\nabla u_1\|^2_{H^{l_q}(\mathscr{T}_h)}+\|\bq_t\|^2_{H^{l_q}(\mathscr{T}_h)})+h^{2l_u-2}(\|u_1\|^2_{H^{l_u}(\mathscr{T}_h)}+\|u_t\|^2_{H^{l_u}(\mathscr{T}_h)})
  \nonumber\\
  &+\int_0^t\big(h^{2l_u-2}\|u_{tt}\|^2_{H^{l_u}(\mathscr{T}_h)} +h^{2l_q}\|\bq_{tt}\|^2_{H^{l_q}(\mathscr{T}_h)} 
 +\|\theta_{\bq_t}\|^2 + \|\theta_{u_t}-\widehat{\theta}_{u_t}\|_{\partial\mathscr{T}_h}^2 \nonumber\\
 &+\|\theta_{u_{tt}}\|^2+\|\br_{u_t}\|^2\big)\;ds\Big).
\end{align*}
Then, an application of Gr\"{o}nwall's inequality  $ 1\leqslant l_u \leqslant k+2 $ and $ 0\leqslant l_q \leqslant k+1,$  yields
\begin{align}
     \|\bs_{\bq_t}\|^2+\|\theta_{u_{tt}}\|^2+\|\bs_{u_t}-\widehat{\bs}_{u_t}\|_{\partial\mathscr{T}_h}^2\leqslant C\;h^{2(k+1)} \big(\|u_1\|^2_{H^{k+2}}+\|u_t\|^2_{H^{k+2}}
  +\int_0^t\big(\|u_{tt}\|^2_{H^{k+2}} +\|\bq_{tt}\|^2_{H^{k+1}} 
  \big)\;ds\big).\label{qt_estimate}
\end{align}}
This completes  the rest of  the proof.

\subsection{Super-convergence Results}
This section derives the third result  \eqref{super-cgt-NHDG}  of the Theorem \ref{newth}. We use the same dual problem \eqref{ereq3} defined in Section 4 with the same regularity results \eqref{p0} and \eqref{p2}.  Multiplying by $\theta_u$ to \eqref{ep1} and integrating over $\Omega$, then split the time derivative and use \eqref{ep3}-\eqref{ep5}
to find that
\begin{align*}
(\Theta,\theta_u)=  (\phi_s(0),\theta_u(0))-(\phi(0),\theta_{u_s}(0))+\int_0^t\Big(-(\phi,\theta_{u_{ss}}) + \big(\nabla\cdot (\nabla\phi),\theta_u\big)-( f_u(u)\phi,\theta_u)\Big)ds.
\end{align*}
Here, we set $u_h(0)=\pw u_0$ and $u_{ht}(0)=\pw u_1,$ which makes the first and second terms on the right-hand side zero. Using the $L^2$ projection and integration by parts yields
\begin{align*}
(\Theta,\theta_u)=\int_0^t\Big(-(\phi,\theta_{u_{ss}})+ \big(\nabla\cdot (\pV\nabla\phi),\theta_u\big)+\langle(\nabla\phi-\pV\nabla\phi)\cdot\bn,\theta_u\rangle-( f_u(u)\phi,\theta_u)\Big)ds.
\end{align*}
Using \eqref{nee1} with $\bv = \pV\nabla\phi$, we obtain
\begin{align*}
(\Theta,\theta_u)= \int_0^t\Big(-(\phi,\theta_{u_{ss}}) -(\bs_\bq,\pV\nabla\phi) +\langle\widehat{\bs}_u,\pV\nabla\phi\rangle+\langle(\nabla\phi-\pV\nabla\phi)\cdot\bn,\theta_u\rangle-( f_u(u)\phi,\theta_u)\Big)ds.
\end{align*}
The second term on the right-hand side can be rewritten as
\[(\bs_\bq,\pV\nabla\phi)=(\bs_\bq,\pV\nabla\phi-\nabla \pw \phi+\nabla \pw \phi)=(\bs_\bq,\nabla \pw \phi)+(\bs_\bq,\nabla\phi-\nabla \pw \phi).\]
Hence, using  \eqref{nee2} for the term $(\bs_\bq,\nabla \pw \phi)$, we arrive at
\begin{align*}
(\Theta,\theta_u)=& \int_0^t\Big(-(\phi,\theta_{u_{ss}}) +(\bs_{u_{ss}},\pw \phi)-\langle(\bq-\widehat{\bq}_h)\cdot\bn,\pw \phi\rangle+(f(u)-f(u_h),\pw \phi)\\&-(\bs_\bq,\nabla\phi-\nabla \pw \phi)+\langle\widehat{\bs}_u,\pV\nabla\phi\rangle+\langle(\nabla\phi-\pV\nabla\phi)\cdot\bn,\theta_u\rangle-( f_u(u)\phi,\theta_u)\Big)ds,\\
=&Q_1+Q_2+Q_3+Q_4+Q_5+Q_6+Q_7+Q_8.
\end{align*}
{Substituting the identity \eqref{identity} in terms $Q_3$ and $Q_5$, a use of \eqref{nee4} with properties of $L^2$ projection shows
\begin{align*}
\int_0^tQ_3 ds=& \langle(\bq-\widehat{\bq}_h)(0)\cdot\bn,\mathcal{P}\undertilde{\phi}(0)- \pw\undertilde{\phi}(0)\rangle\\&+\langle(\bs_{\bq}-\br_\bq)\cdot\bn-\dfrac{\tau}{h_K}(\mathcal{P}\bs_u-\widehat{\bs}_u)+\dfrac{\tau}{h_K}\mathcal{P}\br_u,\mathcal{P}\undertilde{\phi}- \pw\undertilde{\phi}\rangle.
\end{align*}
And
\begin{align*}
\int_0^tQ_5 ds= \big((\bs_\bq)(0),\nabla\undertilde{\phi}(0)-\nabla \pw \undertilde{\phi}(0)\big)+\int_0^t(\bs_{\bq_s},\nabla\undertilde{\phi}-\nabla \pw \undertilde{\phi})\;ds.
\end{align*}
With the help of the Cauchy Schwarz inequality, the trace inequality, and the properties of $L^2$ projections from \eqref{L2proj}, we now arrive at
\begin{align}
    |\langle(\bs_{\bq}-\br_\bq)\cdot\bn,\mathcal{P}\undertilde{\phi}- \pw\undertilde{\phi}\rangle|&\leqslant C h^{-1/2}(\|\bs_{\bq}\cdot\bn\|+\|\br_\bq\cdot\bn\|)\|\undertilde{\phi}- \pw\undertilde{\phi}\|,\nonumber\\
    &\leqslant C h (\|\bs_{\bq}\cdot\bn\|+\|\br_\bq\cdot\bn\|)\|\undertilde{\phi}\|_2,\label{q31}
\end{align}
similar proceed for the other two terms,
\begin{align}
    \left|\left\langle\dfrac{\tau}{h_K}\mathcal{P}\br_u,\mathcal{P}\undertilde{\phi}- \pw\undertilde{\phi}\right\rangle\right|\leqslant C h^{-3/2}\|\br_u\|\|\undertilde{\phi}- \pw\undertilde{\phi}\|\leqslant C h^{k+2}\|u\|_{H^{k+2}}\|\undertilde{\phi}\|_2,\label{q32}
    \end{align}
    similarly
\begin{align}
    \left|\left\langle\dfrac{\tau}{h_K}(\mathcal{P}\bs_u-\widehat{\bs}_u)\right\rangle\right|\leqslant C h^{k+2}\|u\|_{H^{k+2}}\|\undertilde{\phi}\|_2. \label{q33}
    \end{align}
Combining \eqref{q31}, \eqref{q32} and \eqref{q33} gives the estimate for $Q_3$.
Estimate \eqref{qt_estimate}, help us to find the bound for $Q_5$.
 $Q_6$ and $Q_7$ can be combined by adding the zero term $\langle\widehat{\bs}_u, \nabla \phi\cdot\bn\rangle$ as $\widehat{\bs}_u$ is single-valued across the interior faces and us the identity \eqref{identity}, therefore,
\[Q_6+Q_7 = \langle(\bs_u-\widehat{\bs}_u)(0),(\nabla\undertilde{\phi}(0)-\pV\nabla\undertilde{\phi}(0))\cdot\bn\rangle+\int_0^t \langle\bs_{u_s}-\widehat{\bs}_{u_s},(\nabla\undertilde{\phi}-\pV\nabla\undertilde{\phi})\cdot\bn\rangle\]
\begin{align*}
|\langle\bs_{u_s}-\widehat{\bs}_{u_s},(\nabla\undertilde{\phi}-\pV\nabla\undertilde{\phi})\cdot\bn\rangle| &\leqslant \|\bs_{u_s}-\widehat{\bs}_{u_s}\|\| \nabla (\phi- I_h\undertilde{\phi})\|\; \\
&\leqslant C\;h^{k+2} \big(\|u_1\|^2_{H^{k+2}}+\|u_t\|^2_{H^{k+2}}
  +\int_0^t\|u_{tt}\|^2_{H^{k+2}}\;ds\big)\;\|\undertilde{\phi}\|_{H^2}.
\end{align*} }
For the nonlinear terms, we  rewrite as
\[Q_4+Q_8=(f_u(u)\bs_u,\pw\phi-\phi)+(f_u(u)\br_u,\pw\phi)+(f_{uu}(u)(u-u_h)^2,\pw\phi).\]
To complete the derivation, we proceed by bounding each term systematically.  Using the estimates of $\bs_{u_{ss}}$ from \eqref{cgt-NHDG1}, we bound $Q_1$ and $Q_2$. For $Q_5$, proceed similarly as done in subsection 4.2. At last, combining all the terms, applying regularity results \eqref{p0} and \eqref{p2} with the properties of $L^2$ projections, we derive the required result.  This completes the rest of the proof.

\section{Numerical Experiments}
This section illustrates three numerical examples to highlight our proposed methods' convergence, superconvergence, and energy conservation properties.\\

\noindent\textbf{Example-1}: 
 In this example, we consider the domain $\Omega = [0, 1]\times[0, 1]$ with homogeneous Dirichlet boundary conditions.
 The source terms $f_1(x,y,t)$, and the initial conditions are chosen so that the exact solution $u(x,y,t) = t^2\sin(\pi x)\sin(\pi y)$ for the nonlinear Klein-Gordon equation
\begin{align*}
&u_{tt}-\Delta u+u^3-u=f_1(x,y,t)\text{~~ in }[0, 1]\times[0, 1]\times(0,T],\\
&u|_{t=0}=u_0,~u_t|_{t=0}=u_1\text{~~~~~~~~~~~ in }[0, 1]\times[0, 1].
\end{align*}
The second-order discrete-time Galerkin scheme
is applied for the time discretization and the space discretization; we choose polynomial degrees $k = 1, 2, 3$ with mesh parameters $h=1/2^{m}$, for $m=1, 2, 3, 4$. The time
step is chosen to be $\Delta t = h^{(k+1)/2}$ with final time $T=1$.  Table \ref{table-1} confirms our numerical results with our theoretical findings.\\

\noindent\textbf{Example-2}: With the same problem and domain as in Example 1 and the chosen source term $f_2(x,y,t)$, and the initial conditions so that the exact solution $u=\exp(2t^2)\sin(\pi x)\sin(\pi y).$ In Table \ref{table-2}, we observe that numerical orders of convergence and superconvergence results are the same as have been predicted by our theoretical results.\\ 

\noindent\textbf{Example-3}: Next, we consider $u=\tanh(\frac{x}{\sqrt{3}}-t),$ can be the solution of nonlinear traveling wave. With the same problem and domain as mentioned in Example 1 and the chosen source term $f_2(x,y,t)$, with the initial conditions, we obtain the superconvergence for higher degree polynomials, see Table \ref{table-3}). \\

\noindent\textbf{Example-4}: We consider the following nonlinear wave equation in the unit square: $\Omega = [0, 1]\times[0, 1]$ with source terms $f_1(x,y,t)=0$, with the given initial conditions,
\begin{align*}
&u_{tt}-\Delta u+u^3-u=0\quad\quad\text{~~~~~~~~~~~~~~~~~~~~~~~~~~~~~~~~~~~~~~~~~~~~~~~ in }[0, 1]\times[0, 1]\times(0,T],\\
&u|_{t=0}=20x^2(1-x)^2y^2(1-y)^2,~u_t|_{t=0}=2\sin(2\pi x)\sin(2\pi y)\quad\text{ in }[0, 1]\times[0, 1].
\end{align*}
For this example, we do not have the explicit exact solution; therefore, we use energy conservation property.
We choose polynomial degrees $k = 1$ with mesh parameters $h=1/2^{m}$ for $m=1, 2, 3, 4$. The time
step is chosen to be $\Delta t = 0.1$ and $n =1/\Delta t $ with final time $T=1$. It is observed that our computational results in Table \ref{table-4} are conforming to the energy conservation for $n\geqslant 1$ as has been observed theoretically.
\begin{table}[H]
\begin{center}
\begin{tabular}{||c c c c c c c c||} 
 \hline
    $k$ & $m$ & $\|u-u_h\|$ & &$\|\bq-\bq_h\|$ & & $\|u-u^{*}_{h}\|$ & \\ [0.5ex] 
     & & Error& E.O.C& Error & E.O.C & Error & E.O.C \\
 \hline\hline
 \multirow{4}{*}{1} &  1  & 1.83e-1  &        & 5.44e-1  &        & 4.50e-2 &   \\
    &  2  & 6.46e-2  & 1.5101 & 1.47e-1  &  1.8912 & 5.20e-3 & 3.1017    \\
    &  3  & 1.34e-2  &  2.2697 & 2.98e-2  &  2.2994 & 5.00e-4 & 3.2694   \\
    &  4  & 3.40e-3  &  1.9631 & 7.40e-3  &  2.0021 & 1.00e-4  &  3.0221
    \\ \hline
 \multirow{4}{*}{2}  &  1   & 4.21e-2 &         & 7.74e-2  &        & 5.00e-3 &   \\
    &  2   & 4.80e-3  & 3.1226 & 9.90e-3  & 2.9640 & 3.00e-4 & 4.0556    \\
    &  3   & 6.00e-4  & 3.0169 & 1.20e-3  & 3.0688 & 1.00e-5 & 4.0999   \\
    &  4   & 1.00e-4  & 3.0158 & 1.00e-4  & 3.0253 & 1.00e-6  & 4.0116   \\
 \hline
 \multirow{4}{*}{3} &  1   & 8.60e-3  &        & 2.11e-2  &        & 1.30e-3 &   \\
    &  2   & 5.00e-4  & 4.1152 & 1.20e-3  & 4.1710 & 1.00e-4 & 5.1703   \\
    &  3   & 1.00e-5  & 4.0331 & 1.00e-4  & 4.0807 & 1.00e-5 & 5.0574   \\
    &  4   & 1.00e-6 & 3.9941  & 1.00e-5 & 4.0138 & 1.00e-6  & 5.0117 \\  
 \hline 
\end{tabular}
\caption{\label{table-1} The
$L_{2}$-Errors and Estimated Orders of Convergence for Example 1.}
\end{center}
\end{table}
\begin{table}[H]
\begin{center}
\begin{tabular}{||c c c c c c c c||} 
 \hline
    $k$ & $m$ & $\|u-u_h\|$ & &$\|\bq-\bq_h\|$ & & $\|u-u^{*}_{h}\|$ & \\ [0.5ex] 
     & & Error& E.O.C& Error & E.O.C & Error & E.O.C \\
 \hline\hline
\multirow{4}{*}{1} &  1  & 6.77e-1  &        & 5.23e-0  &        & 6.65e-1 &   \\
    &  2  & 1.83e-1  &  1.7470 & 1.64e-0  &  1.6729 & 5.54e-2 & 2.7870   \\
    &  3  & 9.18e-2  &  2.0836 & 5.43e-1  & 1.7944  & 1.18e-2 & 3.0439   \\
    &  4  & 5.65e-3  &  1.9906 & 1.61e-1  &  1.9502  & 7.93e-3  &  2.9804  \\
 \hline
\multirow{4}{*}{2}  &  1   & 6.52e-1 &         & 1.08e-0  &        & 9.04e-2 &   \\
    &  2   & 3.57e-2  & 4.1904 & 1.89e-1  & 2.5145 & 1.30e-2 & 3.7951    \\
    &  3   & 3.60e-3  & 3.3099 & 3.06e-2  & 2.6259 & 8.00e-4 & 4.0171   \\
    &  4   &5.00e-4  & 2.9537 & 6.30e-3  &  2.8868 & 1.00e-4  & 4.1590 \\
 \hline
\multirow{4}{*}{3} &  1   & 6.28e-2  &        & 5.33e-1  &        & 5.91e-2 &   \\
    &  2   & 3.09e-3  & 4.0203 & 2.54e-2  & 4.3924 & 2.30e-3 & 4.8996   \\
    &  3   & 2.00e-4  & 4.1077& 1.10e-3  & 4.5548 & 1.00e-4 & 4.9844   \\
    &  4   & 1.39e-5 &  4.0046  & 9.12e-5 & 4.5665 & 2.79e-6  & 4.9810 \\  
 \hline
\end{tabular}
\caption{\label{table-2} The
$L_{2}$-Errors and Estimated Orders of Convergence for Example 2.}
\end{center}
\end{table}
\begin{table}[H]
\begin{center}
\begin{tabular}{||c c c c c c c c||} 
 \hline
    $k$ & $m$ & $\|u-u_h\|$ & &$\|\bq-\bq_h\|$ & & $\|u-u^{*}_{h}\|$ & \\ [0.5ex] 
     & & Error& E.O.C& Error & E.O.C & Error & E.O.C \\
 \hline\hline
 \multirow{4}{*}{1} &  1  &4.44e-2  &        & 2.95e-2  &        & 8.89e-4 &   \\
    &  2  & 6.60e-3  & 2.4571 & 6.70e-3  &  2.1305 & 1.60e-4 & 2.8746    \\
    &  3  & 1.20e-3  &  2.4239 & 1.80e-3   &  1.9218 & 2.00e-4 & 2.4726   \\
    &  4  & 1.00e-4  &  2.2456 & 7.00e-4  &  2.0021 & 1.00e-5  &  3.0256
    \\ \hline
 \multirow{4}{*}{2}  &  1   & 6.50e-3 &         & 1.59e-2  &        & 3.50e-4 &   \\
    &  2   & 3.50e-3  & 4.3343 & 1.00e-3  & 3.9765 & 1.09e-5 & 4.3122    \\
    &  3   & 2.00e-4  & 3.9873 & 1.00e-4  & 2.7893 & 1.00e-5 & 3.9956   \\
    &  4   & 1.00e-4  & 3.4589 & 1.00e-5  & 3.0253 & 1.00e-6  & 4.0116   \\
 \hline
 \multirow{4}{*}{3} &  1   & 8.60e-3  &        & 7.60e-3  &        & 1.60e-3 &   \\
    &  2   & 1.60e-4  & 5.1964 & 3.20e-4  & 5.1847 & 1.30e-5 & 5.1916   \\
    &  3   & 1.00e-5  & 5.0087 & 1.00e-4  & 5.0003 & 1.00e-5 & 5.0079   \\
    &  4   & 1.00e-6 & 5.0415  & 1.00e-5 & 5.0012 & 1.00e-6  & 5.0117 \\  
 \hline 
\end{tabular}
\caption{\label{table-3} The
$L_{2}$-Errors and Estimated Orders of Convergence for Example 3.}
\end{center}
\end{table}
\begin{table}[H]
\begin{center}
\begin{tabular}{||c c c c c ||}
 \hline
$n$ & $m=1$ & $m= 2$ & $m=3 $ & $m=4 $\\ [0.5ex] 
 \hline\hline
 1   & 0.6402    &  0.0968  &  0.4813 &    0.3576\\
2  & 1.3878e-17  & 3.5534e-15 &1.7765e-14 &3.1972e-14\\
3   &  1.3878e-17  & 1.7767e-15   &1.7765e-14 &5.3300e-15\\
4   & 1.3878e-17& 5.3297e-15 &1.1421e-14 &3.5500e-15\\
5   & 1.3878e-17& 7.1052e-15 &1.0660e-14 &5.5329e-14\\
6   &  2.7756e-17  &  8.8823e-15 &2.4876e-14 &3.3730e-14\\
7   & 2.7756e-17  & 0.0000e-0   &1.0660e-14 &5.5070e-14\\
8   & 4.1633e-17  & 3.5537e-15&3.5534e-14 &9.2375e-14\\
9   & 4.1633e-17  &  3.5527e-15&2.1324e-14 &3.7340e-15\\
10   & 1.3878e-17  & 0.0000e-0&1.4213e-14 &8.5278e-14\\
\hline
\end{tabular}
\caption{\label{table-4} Error Energy $|\mathbb{E}^{n+1/2}-\mathbb{E}^{1/2}|$ for Example 4.}
\end{center}
\end{table}

\section{Summary and Conclusion}
In this study,  we have developed and analyzed the HDG approach for solving the nonlinear Klein-Gordon equation. The nonlinear function being odd degree polynomials is locally Lipschitz. HDG projection was used to derive the error estimates for the semi-discrete case. Further, element-by-element post-processing has been proposed. It has been demonstrated that the post-processed discrete displacement has super-convergence property, which is of order $k + 2, k \geqslant 1$. In contrast, the discrete displacement and its gradient have optimal rates of convergence, that is, of order $k + 1, k\geqslant0$, when piecewise polynomials of degree $k\geq 1$ are used to approximate both the displacement and its flux variables.
To discretize in time, a second-order conservative finite difference scheme is used. It is shown that discrete energy is conserved for $n\geqslant1$, and a priori error estimates are derived. A non-conservative scheme is also suggested, which computes the discrete solutions via a linear system of algebraic equations at each time level, and its optimal error estimates are briefly addressed. A variant of the HDG method is also analyzed, along with error analysis in brief.  Finally, several computational experiments are conducted, the results of which confirm our theoretical findings.The present analysis can be easily extended  to  problems with odd degree polynomial nonlinearity  or to sine Gordon equation.  

\bibliography{refer_1}
\bibliographystyle{plain}
\end{document}